\documentclass[12pt]{article}
\usepackage{amsmath,amssymb,amsthm,bbm,natbib,color,hyperref}
\usepackage[margin=1.3in]{geometry}

\newtheorem{theorem}{Theorem}
\newtheorem{lemma}{Lemma}
\newtheorem{proposition}{Proposition}
\newtheorem{corollary}{Corollary}

\newcommand{\R}{\mathbb{R}}
\newcommand{\eps}{\epsilon}
\newcommand{\EE}[1]{\mathbb{E}\left[{#1}\right]}
\newcommand{\EEst}[2]{\mathbb{E}\left[{#1}\  \middle| \ {#2}\right]}
\newcommand{\Ep}[2]{\mathbb{E}_{{#1}}\left[{#2}\right]}
\newcommand{\PP}[1]{\mathbb{P}\left\{{#1}\right\}}
\newcommand{\PPst}[2]{\mathbb{P}\left\{{#1}\  \middle| \ {#2}\right\}}
\newcommand{\Ppst}[3]{\mathbb{P}_{{#1}}\left\{{#2}\  \middle| \ {#3}\right\}}
\newcommand{\Pp}[2]{\mathbb{P}_{{#1}}\left\{{#2}\right\}}

\newcommand{\Ch}{\widehat{C}}
\newcommand{\Xcal}{\mathcal{X}}
\newcommand{\Xset}{\mathfrak{X}}
\newcommand{\Bcal}{\mathcal{B}}
\newcommand{\Ccal}{\mathcal{C}}
\newcommand{\One}[1]{{\mathbbm{1}}\left\{{#1}\right\}}
\newcommand{\leb}{\textnormal{leb}}
\renewcommand{\emptyset}{\varnothing}
\newcommand{\Lcal}{\mathcal{L}}
\newcommand{\muh}{\widehat{\mu}}

\newcommand{\qh}{\widehat{q}}
\newcommand{\iidsim}{\stackrel{\textnormal{iid}}{\sim}}
\renewcommand{\O}[1]{\mathcal{O}({#1})}
\newcommand{\vc}{\textnormal{VC}}
\newcommand{\vcae}{\vc_{\textnormal{a.e.}}}

\newcommand*\samethanks[1][\value{footnote}]{\footnotemark[#1]}

\title{The limits of distribution-free conditional predictive inference}
\author{Rina Foygel Barber\thanks{Department of Statistics, University of Chicago} , 
Emmanuel J.~Cand{\`e}s\thanks{Departments of Statistics and Mathematics, Stanford University} ,
 \\ Aaditya Ramdas\thanks{Department of Statistics and Data Science, Carnegie Mellon University} \thanks{Machine Learning Department, Carnegie Mellon University} , 
 Ryan J.~Tibshirani\samethanks[3] \samethanks }

\date{\today}

\begin{document}
\maketitle

\begin{abstract}
We consider the problem of distribution-free predictive inference, with the goal of producing predictive coverage guarantees
that hold conditionally rather than marginally. Existing methods such as conformal prediction offer marginal coverage
guarantees, where predictive coverage holds on average over all possible test points, but this is not sufficient for many practical applications
where we would like to know that our predictions are valid for a given individual, not merely on average over a population.
On the other hand, exact conditional inference guarantees are known to be impossible
without imposing assumptions on the underlying distribution. In this work we aim to explore the space in between these two,
and examine what types of relaxations of the conditional coverage property would alleviate some of the practical concerns with marginal coverage guarantees
while still being possible to achieve in a distribution-free setting. 
\end{abstract}

\section{Introduction}

Consider a training data set $(X_1,Y_1),\dots,(X_n,Y_n)$, and a test point $(X_{n+1},Y_{n+1})$, with the training and test data all drawn i.i.d.~from the same distribution.
Here each $X_i\in\R^d$ is a feature vector, while $Y_i\in\R$ is a response variable.
The problem of {\em predictive inference} is the following: if we observe the $n$ training data points, and are given the feature vector
$X_{n+1}$ for a new test data point, we would like construct a prediction interval for $Y_{n+1}$---that is, a subset of $\R$ that we believe is likely
to contain the test point's true response value $Y_{n+1}$. 

As a motivating example, suppose that each data point $i$ corresponds to a patient,
with $X_i$ encoding relevant covariates (age, family history, current symptoms, etc.), while the response $Y_i$ measures a quantitative outcome 
(e.g., reduction in blood pressure after treatment with a drug). When a new patient arrives at the doctor's office with covariate values $X_{n+1}$, the doctor would like to be
able to predict their eventual outcome $Y_{n+1}$ with a range, making a statement along the lines of: ``Based on your age, family history, and current symptoms,
you can expect your blood pressure to go down by 10--15mmHg''. In this paper, we will study the problem of making accurate predictive statements of this sort.

To study such questions, throughout this paper we will write
$\Ch_n(x)\subseteq\R$ to denote the prediction interval\footnote{Note
  that the set $\Ch_n(x)\subseteq\R$ is not required to be an
  interval---it may consist of a disjoint union of multiple
  intervals. For simplicity we still refer to the \smash{$\Ch_n(x)$}'s
  as ``prediction intervals''.} for $Y_{n+1}$ given a feature vector
$X_{n+1}=x$. This interval is a function of both the test point $x$
and the training data $(X_1,Y_1),\dots,(X_n,Y_n)$. We will write
\smash{$\Ch_n$} (without specifying a test point $x$) to refer to the
algorithm that maps the training data $(X_1,Y_1),\dots,(X_n,Y_n)$ to
the resulting prediction intervals $\Ch_n(x)$ indexed by $x\in\R^d$.
(For convenience in writing our results, we assume that the $X_i$'s
lie in $\R^d$, although our results hold more generally for any
probability space.)
 
For the algorithm \smash{$\Ch_n$} to be useful, we would like to be assured that the resulting prediction interval is indeed likely to contain the true response value,
i.e., that $Y_{n+1}\in\Ch_n(X_{n+1})$ with fairly high probability. When this event succeeds, we say that the predictive interval \smash{$\Ch_n(X_{n+1})$} {\em covers} the true 
response value $Y_{n+1}$. Defining the coverage probability is not a trivial question---do we require that coverage holds with high probability on average over the test feature
vector $X_{n+1}$, pointwise at any value $X_{n+1}=x$, or something in between? In order to be robust to distributional assumptions, we would also like to ensure that
our algorithm \smash{$\Ch_n$} has good coverage properties without making any assumptions about the underlying distribution $P$---a ``distribution-free'' guarantee.

To formalize these ideas, we will begin with a few definitions. Throughout, $P$ will denote a joint distribution on $(X,Y)\in\R^d\times\R$, and we will write $P_X$ to denote
the induced marginal on $X$, and $P_{Y|X}$ for the conditional distribution of $Y|X$.
We say that \smash{$\Ch_n$} satisfies {\em distribution-free marginal coverage} at the level $1-\alpha$, denoted by $(1-\alpha)$-MC, if\footnote{In these definitions, and throughout the remainder of the paper, all probabilities are taken with respect to training data $(X_1,Y_1),\dots,(X_n,Y_n)$ and test point $(X_{n+1},Y_{n+1})$ all drawn i.i.d.~from $P$, unless specified otherwise. }
\begin{equation}\label{eqn:marginal_cov}
\PP{Y_{n+1}\in \Ch_n(X_{n+1})} \geq  1-\alpha\textnormal{ for all distributions $P$}.
\end{equation}
In other words, the probability that \smash{$\Ch_n$} covers the true test value $Y_{n+1}$ is at least $1-\alpha$, on average over a random draw of the training and test data from {\em any} distribution $P$.
We say that \smash{$\Ch_n$} satisfies {\em distribution-free conditional coverage} at the  level $1-\alpha$, denoted by $(1-\alpha)$-CC, if
\begin{equation}\label{eqn:conditional_cov}
\PPst{Y_{n+1}\in \Ch_n(X_{n+1})}{X_{n+1}=x} \geq  1-\alpha\textnormal{ for all $P$ and almost all $x$},
\end{equation}
where, fixing the distribution $P$, we write ``almost all $x$'' to mean that the set of points $x\in\R^d$ where the bound fails to hold must have measure zero under $P_X$. This means that the probability that \smash{$\Ch_n$} covers, at a {\em fixed} test point $X_{n+1}=x$, is at least $1-\alpha$.\footnote{\citet{vovk2012conditional} also considers a notion of conditional coverage,
where the guarantee is required to hold after conditioning on the training data $(X_1,Y_1),\dots,(X_n,Y_n)$ but without conditioning on the test point $X_{n+1}$,
and thus is very different from the type of conditioning that we consider here.}

Now, how should we interpret the difference between marginal and conditional coverage? With $\alpha=0.05$,
we expect that the doctor's statement (``...you can expect your blood pressure to go down by 10--15mmHg'') should hold with 95\% probability.
For marginal coverage, the probability is taken over both $X_{n+1}$ and $Y_{n+1}$, while for conditional coverage, $X_{n+1}$ is fixed and the probability
is taken over $Y_{n+1}$ only (and over all the training data in both situations). This means that for marginal coverage, the doctor's statements have a 95\% chance
of being accurate {\em on average} over all possible patients that might arrive at the clinic (marginalizing over $X_{n+1}$), but might for example have 0\% chance of being
accurate for patients under the age of 25, as long as this is averaged out by a higher-than-95\% chance of coverage for patients older than 25. 
The stronger definition of conditional coverage, on the other hand, removes this possibility, and requires that
whatever statement the doctor makes (different for each patient) has a 95\% chance of being true for every individual patient, regardless of the patient's age, family history, etc.

For practical purposes, then, marginal coverage does not seem to be sufficient---each patient would reasonably hope that the information they receive
is accurate for their specific circumstances, and is not comforted by knowing that the inaccurate information they might be receiving will be balanced out by some other patient's highly precise
prediction. On the other hand, the problem of conditional inference is statistically very challenging, and is known to be incompatible with the distribution-free setting (we will 
discuss this in more detail later on). Our goal in this paper is therefore to explore the middle ground between marginal and conditional inference, while working in the
distribution-free setting in order to be robust to violations of any modeling assumptions.

\subsection{Summary of contributions}\label{sec:summary}
As mentioned above, it is known to be impossible for any finite-length prediction interval
to satisfy distribution-free conditional coverage in the sense of~\eqref{eqn:conditional_cov}---this
is because, without assuming smoothness of the underlying distribution $P$, we cannot exclude the possibility that there is 
some sort of discontinuity at $X=x$ that leads to a failure of coverage. (Background on this type of impossibility result is described
more formally in Section~\ref{sec:background_CC_impossible}.)

This impossibility motivates us to consider an approximate version of the conditional coverage property.
We will say that \smash{$\Ch_n$} satisfies {\em distribution-free approximate conditional coverage} at level $1-\alpha$ and tolerance $\delta>0$, denoted by $(1-\alpha,\delta)$-CC, if
\begin{multline}\label{eqn:approx_conditional_cov}
\PPst{Y_{n+1}\in \Ch_n(X_{n+1})}{X_{n+1}\in\Xcal} \geq  1-\alpha\textnormal{ for all distributions $P$}\\\textnormal{ and all $\Xcal\subseteq\R^d$ with $P_X(\Xcal)\geq \delta$.}
\end{multline}
For example, at $\alpha = 0.05$ and $\delta=0.1$,
the coverage probability has to be at least 95\% for any subgroup of patients that makes up at least 10\% of the overall population.
If $\delta>0$ is fairly small, then this approximate conditional coverage property
 is quite a bit stronger than marginal coverage, and may be sufficient for many applications.
 
 However, we find that it is inherently impossible to
 find non-trivial algorithms that achieve even this relaxed notion of conditional coverage. 
 Specifically, we compare against a trivial solution: 
we show with a simple argument that any method $\Ch_n$ that satisfies $(1-\alpha\delta)$-MC, will also
satisfy $(1-\alpha,\delta)$-CC. In this sense, we can trivially achieve approximate conditional coverage by way of marginal
coverage, but this solution is not satisfactory since, for small $\delta$, a $(1-\alpha\delta)$-MC prediction interval
will be extremely wide. 
However,
the main result of this paper, Theorem~\ref{thm:main} (see Section~\ref{sec:approximate_CC}),
proves that any $(1-\alpha,\delta)$-CC method is essentially no better than this kind of trivial construction
(in the sense of the expected length of the resulting intervals).

Perhaps, then, the definition~\eqref{eqn:approx_conditional_cov} of approximate conditional coverage may be stronger
than needed in practical applications. In a medical setting, for instance, a patient would typically want to know that coverage is accurate
on average over {\em a subgroup of patients similar to the individual}, and would not be concerned about arbitrary subgroups
consisting of highly dissimilar patients. This motivates us to consider alternatives to 
the approximate conditional coverage property~\eqref{eqn:approx_conditional_cov}---in Section~\ref{sec:restricted_CC},
we modify~\eqref{eqn:approx_conditional_cov} to consider only a restricted class of sets $\Xcal$, for instance,
only sets consisting of balls under some metric (to represent patients similar to the individual of interest, in our example).
We construct an example of an algorithm that satisfies this type of property---a modification of the split conformal method---that we
analyze in Theorem~\ref{thm:Xset}. We also
 establish lower (Theorem~\ref{thm:Xset_VC_lowerbd}) and upper (Theorem~\ref{thm:Xset_VC_upperbd}) bounds on the efficiency of any predictive method
 satisfying this type of property,
as a function of the complexity (VC dimension) of the class of sets over which coverage is required to hold.

\subsection{Notation} Before proceeding, we establish some notation and terminology that will be used
throughout the paper. All sets and functions are  implicitly assumed to be measurable
(e.g., ``for all $\Xcal\subseteq\R^d$'' in~\eqref{eqn:approx_conditional_cov} should be interpreted
to mean all measurable subsets of $\R^d$). The function $\leb()$ denotes Lebesgue measure on $\R$ or on $\R^d$.
Prediction intervals are allowed to be either fixed or randomized. Specifically, a non-data-dependent prediction interval $C = C(x)$
may either be fixed (i.e., a function mapping points $x\in\R^d$ to subsets $C(x)\subseteq\R$)
or random (i.e., a function mapping points $x\in\R^d$ to a random variable $C(x)$ taking values in the set of subsets of $\R$).
Analogously, for a data-dependent prediction interval $\Ch_n = \Ch_n(x)$, 
fixing the training data $(X_1,Y_1),\dots,(X_n,Y_n)$ and the vector $x\in\R^d$, this interval may be either a fixed or random subset of $\R$.

\section{Background}
In this section, we give background on the split conformal prediction method, which achieves distribution-free marginal
coverage, and review results in the literature establishing that distribution-free conditional coverage is not possible.

\subsection{Split conformal prediction}\label{sec:background_split_conformal}

The {\em split conformal prediction} algorithm, introduced in \citet{papadopoulos2002inductive,vovk2005algorithmic} (under the name ``inductive conformal prediction'') and studied further by \citet{papadopoulos2008inductive,vovk2012conditional,lei2018distribution}, is a well known method that achieves distribution-free marginal coverage guarantees. This method makes no assumptions at all on the distribution of the data aside from requiring that the training data and the test point are exchangeable. (Of course, assuming that the training and test data are i.i.d.~is simply a special case of the exchangeability assumption.) 

The split conformal prediction method begins by partitioning the sample size $n$ into two portions, $n = n_0+n_1$, e.g., split in half. We will use the first $n_0$ many training points
to fit an estimated regression function $\muh_{n_0}(x)$, and the remaining $n_1=n-n_0$ many training points to determine the width of the prediction interval around $\muh_{n_0}(x)$. The estimated model $\muh_{n_0}$ can be fitted from $(X_1,Y_1),\dots,(X_{n_0},Y_{n_0})$ using
any algorithm---for example, we might fit a linear model, $\muh_{n_0}(x) = x^\top\widehat\beta$ where $\widehat\beta\in\R^d$ is fitted
on the data points $(X_1,Y_1),\dots,(X_{n_0},Y_{n_0})$ using least squares regression or any other regression method.

Next, fix a desired predictive coverage level $1-\alpha$, for instance 95\%. We then compute residuals
\[R_i = \big| Y_i - \muh_{n_0}(X_i)\big|\textnormal{ for $i=n_0+1,\dots,n$},\]
and define\footnote{Formally, when we write ``the $k$-th smallest value of the list ...'' for a list that has $m$ elements, this will denote $+\infty$ in the case that $k>m$.}
\[\qh_{n_1} = \textnormal{ the $\left\lceil (1-\alpha)(n_1 + 1)\right\rceil$-smallest value of the list $R_{n_0+1},\dots,R_n$.}\]
The predictive interval is then defined as
\begin{equation}\label{eqn:split_conformal_define}\Ch_n(x) = \big[\muh_{n_0}(x) - \qh_{n_1}, \muh_{n_0}(x) + \qh_{n_1}\big].\end{equation}
This method can also be generalized to include a local variance/scale estimate, or to allow for an asymmetric
construction treating the right and left tails of the residuals separately.

The split conformal algorithm is a variant of {\em conformal prediction}, which has a rich literature dating back many years (see, e.g., \citet{vovk2005algorithmic,shafer2008tutorial}
for background). Conformal prediction similarly relies on the exchangeability of the training and test data, but rather than splitting the training data to separate the tasks of model fitting and calibrating the quantiles, conformal prediction uses the full training sample for both tasks, thus paying a higher computational cost. Here, for simplicity, we do not describe conformal prediction, but focus on the split conformal algorithm, which we  generalize in our own proposed methods later on.

Using the assumption that the data points are i.i.d., the proof that the split conformal prediction method satisfies $(1-\alpha)$-MC
is very intuitive. For completeness we state this known result
here.
\begin{theorem}[{\citet[Proposition 1]{papadopoulos2002inductive}}]\label{thm:split_conformal}
The split conformal prediction method defined in~\eqref{eqn:split_conformal_define} satisfies the $(1-\alpha)$-MC property~\eqref{eqn:marginal_cov}.
\end{theorem}
\noindent Importantly, the above guarantee holds irrespective of the regression algorithm used to fit $\muh_{n_0}$. 
Furthermore, \citet{lei2018distribution} show that, in some settings, this distribution-free construction may result in an interval that is asymptotically no wider
than the best possible ``oracle'' interval---in other words, it is possible to provide marginal distribution-free prediction
without incurring a cost in terms of overly wide intervals.
(The intuition behind the proof of Theorem~\ref{thm:split_conformal} will be discussed in Section~\ref{sec:split_Xset} as a special case of our new results;
\citet{lei2018distribution}'s guarantee of optimal length will be discussed in more detail in Section~\ref{sec:location_family}.)

\subsection{Impossibility of distribution-free conditional coverage}\label{sec:background_CC_impossible}

While the split conformal method satisfies distribution-free marginal coverage~\eqref{eqn:marginal_cov}, as mentioned earlier, this property may not be sufficient for practical prediction tasks, as it leaves open the possibility that entire regions of test points (e.g., subgroups of patients) are receiving inaccurate predictions. To avoid this problem, we may wish to construct \smash{$\Ch_n$} to guarantee coverage conditional on $X_{n+1}$, rather than on average over $X_{n+1}$. 
Is it possible to achieve distribution-free conditional coverage~\eqref{eqn:conditional_cov}, 
while still constructing predictive intervals that are not too much larger than needed?

Unfortunately, it is well known that, if we do not place any assumptions on $P$, then estimation and inference on various functionals of $P$ are impossible to carry out; see, e.g., \citet{bahadur1956nonexistence,donoho1988one} for background.
More specifically, for the current problem of  distribution-free conditional prediction intervals,
\citet{vovk2012conditional,lei2014distribution} prove that the $(1-\alpha)$-CC property~\eqref{eqn:conditional_cov}
is impossible for any algorithm
\smash{$\Ch_n$}, unless \smash{$\Ch_n$} has the property that it produces intervals with infinite expected length under {\em any} non-discrete distribution $P$, which is not a meaningful procedure.
\begin{proposition}\label{prop:exact_CC_infinite}[Rephrased from \citet{vovk2012conditional,lei2014distribution}]
Suppose that \smash{$\Ch_n$} satisfies $(1-\alpha)$-CC~\eqref{eqn:conditional_cov}. Then for all distributions $P$, it holds that
\[\EE{\leb(\Ch_n(x))} = \infty\]
at almost all points $x$ aside from the atoms of $P_X$.
\end{proposition}
\noindent In other words, at almost all nonatomic points $x$, the prediction interval has infinite expected length. This means that distribution-free conditional coverage in the sense of~\eqref{eqn:conditional_cov} is impossible to attain in any meaningful sense.

\paragraph{Asymptotic conditional coverage.}
There is an extensive literature examining this problem in a setting where $P$ is assumed to satisfy some type of smoothness
condition, and conditional coverage can then be achieved asymptotically
 by letting the sample size $n$ tend to infinity and using a  vanishing bandwidth  to compute local
smoothed estimators of the conditional distribution of $Y|X$. Works in this line of the literature
include~\citet{cai2014adaptive,lei2014distribution}, among many others. 
In this present work, however, we are interested in obtaining distribution-free guarantees that hold at any finite sample size $n$, and therefore
we aim to avoid relying on assumptions such as smoothness of $P$ or on asymptotic arguments.

\section{Approximate conditional coverage}\label{sec:approximate_CC}
While the results of \citet{vovk2012conditional} and \citet{lei2014distribution} prove that distribution-free methods cannot achieve conditional predictive guarantees, in practice
it may be sufficient to obtain ``approximately conditional'' inference.
In our doctor/patient example, we
would certainly want to make sure that there is no entire subgroup of patients that are all receiving poor predictions---as in our earlier example where the predictive
intervals had poor coverage for all patients below the age of 25---but we may be willing to accept that some rare groups of patients might be receiving inaccurate information.

We therefore try to relax our requirement of conditional coverage to an approximate version---recall from Section~\ref{sec:summary}
that \smash{$\Ch_n$} satisfies {\em distribution-free approximate
  conditional coverage} at level $1-\alpha$ and tolerance $\delta>0$,
denoted by $(1-\alpha,\delta)$-CC, if \eqref{eqn:approx_conditional_cov} holds.
We can easily verify that approximate conditional coverage limits to conditional coverage by taking $\delta$ to zero:
\[\textnormal{\smash{$\Ch_n$} satisfies $(1-\alpha)$-CC} \quad \Longleftrightarrow \quad \textnormal{\smash{$\Ch_n$} satisfies $(1-\alpha,\delta)$-CC for all $\delta>0$}.\]
At the other extreme, marginal coverage is recovered by taking $\delta=1$:
\[\textnormal{\smash{$\Ch_n$} satisfies $(1-\alpha)$-MC} \quad \Longleftrightarrow \quad \textnormal{\smash{$\Ch_n$} satisfies $(1-\alpha,\delta)$-CC for $\delta=1$}.\]

While we have seen that exact conditional coverage is impossible to meaningfully attain, does this relaxation allow us to move towards a meaningful solution? 
To answer this question, it is useful to first consider a simple solution  obtained by way of a marginal coverage method.

\subsection{The inadequacy of reducing to marginal coverage}

The following lemma suggests that our approximate conditional coverage can be naively obtained via marginal coverage at a more stringent level. 

\begin{lemma}\label{lem:alphadelta_MC}
Let \smash{$\Ch_n$} be any method that attains distribution-free marginal coverage~\eqref{eqn:marginal_cov} with miscoverage rate $\alpha\delta$ in place of $\alpha$, that is, \smash{$\Ch_n$} satisfies the $(1-\alpha\delta)$-MC property. Then \smash{$\Ch_n$} also satisfies $(1-\alpha,\delta)$-CC.
\end{lemma}
\begin{proof}[Proof of Lemma~\ref{lem:alphadelta_MC}]
Since \smash{$\Ch_n$} satisfies $(1-\alpha\delta)$-MC, for any distribution $P$ we have
\begin{multline*}
\alpha\delta
\geq \PP{Y_{n+1}\not\in \Ch_n(X_{n+1})}
\geq \PP{Y_{n+1}\not\in \Ch_n(X_{n+1}),X_{n+1}\in\Xcal}\\
\geq \delta\cdot \PPst{Y_{n+1}\not\in \Ch_n(X_{n+1})}{X_{n+1}\in\Xcal},
\end{multline*}
where the last step holds for any $\Xcal$ with $\PP{X_{n+1}\in\Xcal} = P_X(\Xcal)\geq \delta$. Rearranging yields the lemma.
\end{proof}

\noindent To interpret this lemma, we might apply the split conformal prediction algorithm~\eqref{eqn:split_conformal_define} at the miscoverage level $\alpha\delta$, which ensures marginal coverage at this level and, therefore, ensures $(1-\alpha,\delta)$-CC. However, we would typically choose $\delta$ to be quite small, as we would like to be able to condition on small sets $\Xcal$ (to ensure that there aren't any large subgroups of patients all receiving poor information). This means that any prediction intervals satisfying $(1-\alpha\delta)$-MC must generally be extremely wide, e.g., $99.5\%$-coverage intervals instead of $95\%$-coverage intervals when $\alpha = 0.05$ and $\delta=0.1$. Therefore, the naive solution of using marginal coverage to ensure approximate conditional coverage is not satisfactory. 

Before moving on, we extend Lemma~\ref{lem:alphadelta_MC} to generalize the naive solution given by $(1-\alpha\delta)$-MC:
\begin{lemma}\label{lem:alphadelta_MC_extend}
Let \smash{$\Ch_n$} be any method that satisfies $(1-c\alpha\delta)$-MC~\eqref{eqn:marginal_cov}, for some $c\in[0,1]$. Let $\Ch'_n$ be defined as follows: at a test point $x$, with probability $\frac{1-\alpha}{1-c\alpha}$, we define $\Ch'_n(x)=\Ch_n(x)$, or otherwise, we define $\Ch'_n(x)=\emptyset$ (the empty set), where we assume that this decision is carried out independently of $x$ and of the training data. Then $\Ch'_n$ also satisfies $(1-\alpha,\delta)$-CC.
\end{lemma}
\noindent  Proofs for this lemma and for all subsequent theoretical results are given in the Appendix.

To understand the role of the parameter $c$ in this lemma, we can consider
the two extremes---setting $c=1$, we would simply output the interval \smash{$\Ch_n(x)$} that satisfies $(1-\alpha\delta)$-MC, i.e., we return to the naive solution of Lemma~\ref{lem:alphadelta_MC}.
At the other extreme, if we set $c=0$, at any test point $X_{n+1}=x$ the resulting prediction interval would be given by $\R$ with probability $1-\alpha$, or $\emptyset$ otherwise---this
clearly satisfies $(1-\alpha,\delta)$-CC (and, in fact, $(1-\alpha)$-CC)  but is of course meaningless as it reveals no information about the data.

\subsection{Hardness of approximate conditional coverage}
We now introduce our main result, which proves that, as in the exact conditional coverage setting,
the relaxation to $(1-\alpha,\delta)$-conditional coverage is still impossible to attain meaningfully. In particular, the naive solution---obtaining $(1-\alpha,\delta)$-CC by way of marginal coverage,
as in Lemmas~\ref{lem:alphadelta_MC} and~\ref{lem:alphadelta_MC_extend}---is in some sense the best possible method, in terms of the lengths of the resulting prediction intervals. 

To quantify this, for any $P$ and any marginal coverage level $1-\alpha$, consider finding the prediction interval $C_P(x)$ with the shortest possible length, 
subject to requiring marginal coverage to be at least $1-\alpha$ under the distribution $P$. As the notation suggests, the coverage properties of 
$C_P(x)$ are specific to $P$ and are not distribution-free in any sense.
Formally, we define the set of intervals with marginal coverage under $P$ as
\[\Ccal_P(1-\alpha) = \Big\{C_P : \Pp{P}{Y\in C_P(X)}\geq 1-\alpha\Big\},\]
where $C_P(x)$ may denote a fixed or random interval (that is, $C_P$ is a function mapping points $x\in\R^d$ to  fixed or random subsets of $\R$).
We can then define the minimum possible length as
\begin{equation}\label{eqn:def_L_P}L_P(1-\alpha) = \inf_{C_P\in\Ccal_P(1-\alpha)}\Big\{\Ep{P_X}{\leb(C_P(X))}\Big\}.\end{equation}
If $C_P$ is random rather than fixed, then we should interpret the expectation as being taken with respect to the random draw of $X$
and the randomization in the construction of $C_P(X)$.

With these definitions in place, we present our main result, which proves a lower bound on the prediction interval width
of any method that attains distribution-free approximate conditional coverage.
\begin{theorem}\label{thm:main}
Suppose that \smash{$\Ch_n$} satisfies $(1-\alpha,\delta)$-CC~\eqref{eqn:approx_conditional_cov}. Then for all distributions $P$ where the marginal distribution $P_X$ has no atoms, 
\[\EE{\leb(\Ch_n(X_{n+1}))} \geq \inf_{c\in[0,1]} \left\{\frac{1-\alpha}{1-c\alpha}\cdot L_P(1-c\alpha\delta)\right\}.\]
\end{theorem}
\noindent How should we interpret this lower bound? 
Based on Lemma~\ref{lem:alphadelta_MC}, 
we can achieve $(1-\alpha,\delta)$-CC trivially by running split conformal prediction at the marginal coverage level $1-\alpha\delta$.
What would be the average width from such a procedure? As mentioned in Section~\ref{sec:background_split_conformal},
under certain assumptions on $P$,
\citet{lei2018distribution} prove that the split conformal method run at coverage level $1-\alpha\delta$ with a consistent regression algorithm $\muh$ will, with high probability,
output a prediction interval with width that is only $o(1)$ larger than the oracle interval, which has width $L_P(1-\alpha\delta)$.
More generally, for any $c\in[0,1]$, we can use the construction suggested in Lemma~\ref{lem:alphadelta_MC_extend}
combined with the split conformal method, now run at level $1-c\alpha\delta$, to instead
produce expected length $\approx \frac{1-\alpha}{1-c\alpha}\cdot L_P(1-c\alpha\delta)$.

Since Theorem~\ref{thm:main} demonstrates that any method satisfying $(1-\alpha,\delta)$-CC cannot beat this lower bound,
this means that the $(1-\alpha,\delta)$-CC property is impossible to attain beyond the trivial solution, i.e., by applying a method
that guarantees $(1-\alpha\delta)$-marginal coverage, which then yields $(1-\alpha,\delta)$-CC as a byproduct (or choosing some $c\in[0,1]$
for the more general construction). 
Since typically we would choose $\delta$ to be a small constant, this lower bound is indeed a substantial issue,
since $L_P(1-\alpha\delta)$ will generally be much larger than the length we would need if the distribution $P$ were known.

\section{Restricted conditional coverage}\label{sec:restricted_CC}
Our main result, Theorem~\ref{thm:main}, shows that our definition of approximate conditional coverage in~\eqref{eqn:approx_conditional_cov} is too strong; it is
impossible to construct a meaningful procedure  satisfying this definition.
One way to weaken this condition is to restrict which sets $\Xcal$ we consider, yielding a less stringent notion
of approximate conditional coverage. 

For example, we can require that the coverage guarantee
holds ``locally'', by conditioning only on any {\em ball} with sufficient probability $\delta$, rather than on an arbitrary subset $\Xcal\subseteq\R^d$. 
More concretely, we might require that
\begin{multline}\label{eqn:local_conditional_cov}
\PPst{Y_{n+1}\in \Ch_n(X_{n+1})}{X_{n+1}\in\mathbb{B}(x,r)} \geq 1- \alpha\textnormal{ for all distributions $P$}\\\textnormal{ and all $x\in\R^d, r\geq 0$ with $\Pp{P_X}{X\in\mathbb{B}(x,r)} \geq \delta$.}
\end{multline}
Here $\mathbb{B}(x,r)$ is the closed $\ell_2$ ball centered at $x$ with radius $r$. In the doctor/patient example, 
we can think of this as requiring 95\% predictive accuracy on average over the subgroup of population consisting of patients similar to a given patient $x$, where
similarity is defined with the $\ell_2$ norm (of course, we can also generalize this to different metrics).
As another example, \citet{vovk2012conditional,lei2014distribution} consider a version of conformal prediction  that guarantees coverage within each one of a finite number of subgroups, i.e.
\begin{multline}\label{eqn:local_conditional_cov}
\PPst{Y_{n+1}\in \Ch_n(X_{n+1})}{X_{n+1}\in\Xcal_k} \geq1- \alpha\textnormal{ for all distributions $P$}\\\textnormal{ and for all $k=1,\dots,K$,}
\end{multline}
for some fixed partition $\R^d = \Xcal_1\cup \dots \cup \Xcal_K$ of the feature space. Here we may think of predefining subgroups of patients 
(males below age 25, males age 25--35, etc.) and requiring  95\% predictive accuracy on average over each predefined subgroup. 

More generally, suppose we are given a collection $\Xset$ of measurable subsets of $\R^d$. 
We say that
\smash{$\Ch_n$} satisfies distribution-free approximate conditional coverage at level $1-\alpha$ and tolerance $\delta>0$ relative
to the collection $\Xset$, denoted by $(1-\alpha,\delta,\Xset)$-CC, if
\begin{multline}\label{eqn:approx_conditional_cov_Xset}
\PPst{Y_{n+1}\in \Ch_n(X_{n+1})}{X_{n+1}\in\Xcal} \geq 1-\alpha\textnormal{ for all distributions $P$}\\\textnormal{ and all $\Xcal\in\Xset$ with $P_X(\Xcal) \geq \delta$.}
\end{multline}
To avoid degenerate scenarios, we will assume that we always have $\R^d \in\Xset$, meaning that requiring $(1-\alpha,\delta,\Xset)$-CC is always at least as strong as requiring $(1-\alpha)$-MC. Of course, this definition yields the original $(1-\alpha,\delta)$-CC condition if we take $\Xset$ to be the collection of all measurable
sets. If the class $\Xset$ is too rich, then, our main result in Theorem~\ref{thm:main} proves that $(1-\alpha,\delta,\Xset)$-CC is impossible
to achieve beyond trivial solutions. We may ask then whether it's possible to construct
meaningful prediction intervals when $\Xset$ is sufficiently restricted.

In the following, we will first construct a concrete algorithm, based on the split conformal prediction method, that
attains $(1-\alpha,\delta,\Xset)$-CC. Afterwards, we will attempt to determine how the complexity of the class $\Xset$ determines whether 
this algorithm provides meaningful prediction intervals (i.e., narrower intervals than the lower bound of Theorem~\ref{thm:main}), and indeed if this is possible to attain with any algorithm.

\subsection{Split conformal  for restricted conditional coverage}\label{sec:split_Xset}
As a concrete example, we will construct a variant of the split conformal
prediction method, and will generalize \citet{lei2018distribution}'s results on the efficiency of split conformal prediction to establish conditions under which the resulting prediction intervals are 
asymptotically efficient.

Let $\muh_{n_0}(x)$ be some fitted regression function, which estimates the conditional mean of $Y$ given $X=x$. As before, we require that $\muh_{n_0}$ is fitted on the first $n_0$ training samples, $(X_1,Y_1),\dots,(X_{n_0},Y_{n_0})$.
Next, define the residual
\[R_i = \big|Y_i - \muh_{n_0}(X_i)\big|\]
on the remaining training samples $i=n_0+1,\dots,n$ and on the test point $i=n+1$.
(As for the original split conformal method, this procedure can be generalized to include a local scale estimate, $\widehat{\sigma}_{n_0}(X_i)$, 
or to allow for an asymmetric interval that treats the right and left tails of the residuals differently,
but we do not include these generalizations here.)

The original split conformal method operates by observing that the test point residual, $R_{n+1}$, is equally likely to occur
anywhere in the ranked list of residuals $R_{n_0+1},\dots,R_n,R_{n+1}$, i.e., the test residual is exchangeable with the $n_1$ many residuals from 
the held-out portion of the training data. The split conformal prediction interval~\eqref{eqn:split_conformal_define}
is then constructed as
\[\Ch_n(x) = \big[\muh_{n_0}(x) - \qh_{n_1}, \muh_{n_0}(x) + \qh_{n_1}\big],
\]
where $\qh_{n_1}$ is the
$\left\lceil (1-\alpha)(n_1 + 1)\right\rceil$-smallest value amongst
$R_{n_0+1},\dots,R_n$.  The width of this prediction interval is
determined by this residual quantile $\qh_{n_1}$, which is calculated
by pooling all residuals from the holdout set $i=n_0+1,\dots,n$ and is
therefore calibrated to give the appropriate coverage level {\em on
  average} over the distribution $P$.

We now need to modify this construction to guarantee a stronger notion of coverage---we need
to ensure coverage on average over any $\Xcal\in\Xset$ with $P_X(\Xcal)\geq \delta$.
We will need to modify the width of the prediction interval---for example, for a set $\Xcal$ where residuals tend to be large (i.e., $|Y-\muh_{n_0}(X)|$
is likely to be large if we condition on $X\in\Xcal$), the split conformal interval constructed above is too narrow to achieve $1-\alpha$ coverage
on average over this set. We will therefore construct a new interval,
\begin{equation}\label{eqn:PI_Xset}
\Ch_n(x) = \left[\muh_{n_0}(x) -  \qh_{n_1}(x),\muh_{n_0}(x) +  \qh_{n_1}(x)\right].\end{equation}
The width of the interval is now defined locally by the quantity $\qh_{n_1}(x)$, which we will address next. Intuitively, 
if $x$ belongs to a set $\Xcal$ within which residuals tend to be large, we will need $\qh_{n_1}(x)$ to be large in order to
achieve the right coverage level on average over $\Xcal$.

We now construct $\qh_{n_1}(x)$. First, we will
  narrow down the class of subsets to consider. Define
  \[\widehat{N}_{n_1}(\Xcal)=\sum_{i=n_0+1}^n\One{X_i\in\Xcal},\]
the number of holdout points that lie in $\Xcal$.
Next, let
\[\widehat{\Xset}_{n_1} = \left\{\Xcal\in\Xset : \widehat{N}_{n_1}(\Xcal) \geq \delta n_1\left(1  - \sqrt{\frac{2\log(n_1)}{\delta n_1}}\right)\right\} \quad \subseteq \Xset.\]
This definition ensures that, if a given subset $\Xcal$ has probability $\geq \delta$ under $P$, then we will include $\Xcal\in\widehat{\Xset}_{n_1}$ with high probability.
Next let
\begin{multline*}\qh_{n_1}(\Xcal) =  \textnormal{ the $\left\lceil\left(1-\alpha+\frac{1}{n_1}\right)\cdot \big(\widehat{N}_{n_1}(\Xcal)+1\big)\right\rceil$-th smallest value}\\\textnormal{ of }\big\{R_i:n_0+1\leq i\leq n, X_i\in\Xcal\big\}.\end{multline*}
Finally, we set
\begin{equation}\label{eqn:define_qh}\qh_{n_1}(x) = \sup_{\Xcal\in\widehat{\Xset}_{n_1}: x \in\Xcal} \qh_{n_1}(\Xcal) .\end{equation}
(Recall that $\R^d \in\Xset$ by assumption, and thus $\R^d \in \widehat\Xset_{n_1}$,  so there is always at least one set $\Xcal$ in this supremum.)

Our next result proves that this construction achieves the desired approximate conditional coverage property.
\begin{theorem}\label{thm:Xset}
For any class $\Xset$ of measurable subsets of $\R^d$, the prediction interval defined in~\eqref{eqn:PI_Xset}
satisfies $(1-\alpha,\delta,\Xset)$-CC~\eqref{eqn:approx_conditional_cov_Xset}.
\end{theorem}

Of course, the supremum defined in~\eqref{eqn:define_qh} may be impossible to compute efficiently---this will naturally depend
on the structure of the class $\Xset$. (We expect that for simple cases, such as taking $\Xset$ to be the set of all $\ell_2$ balls
as for the ``local'' conditional coverage discussed earlier, we may be able to compute or approximate~\eqref{eqn:define_qh}  more efficiently; we leave
this as an open question for future work.)
Furthermore,
this guarantee does not yet establish that this method provides a meaningful prediction interval---it may be the case
that the intervals are too wide. We will examine this question next.

\subsection{Characterizing hardness with the VC dimension}
For a class $\Xset$ of subsets of $\R^d$, we write $\vc(\Xset)$ to denote
 the Vapnik--Chervonenkis dimension of the class $\Xset$. This measure of complexity is defined as follows.
For any finite set $\mathcal{A}$ of points in $\R^d$, we say that $\mathcal{A}$ is {\em shattered} by $\Xset$ if, for every subset of points $\Bcal\subseteq \mathcal{A}$, there exists some $\Xcal\in\Xset$
with $\Xcal\cap \mathcal{A} = \Bcal$. The VC dimension is then defined as
\[\vc(\Xset) = \max\left\{|\mathcal{A}| : \textnormal{$\mathcal{A}$ is shattered by $\Xset$}\right\},\]
i.e., the largest cardinality of any set shattered by $\Xset$. Well known examples include:
\begin{itemize}
\item If $\Xset$ is the set of all $\ell_2$ balls in $\R^d$, then $\vc(\Xset)=d+1$.
\item If $\Xset$ is the set of all half-spaces in $\R^d$, then $\vc(\Xset)=d+1$.
\item If $\Xset$ is the set of all intersections of $k$ different half-spaces in $\R^d$, then $\vc(\Xset)=\O{kd\log(k)}$~\citep[Lemma 3.2.3]{blumer1989learnability}.
\end{itemize}

While a large VC dimension of $\Xset$ ensures that there is {\em some} set of points $\mathcal{A}$ that is shattered by $\Xset$,
we need a stronger formulation to establish a hardness result for restricted conditional coverage. 
We will consider an ``almost everywhere'' version of the VC dimension, defined as follows:
\[\vcae(\Xset) = \max\left\{m\geq 0: \begin{array}{c}\textnormal{the class of sets $\mathcal{A}=\{a_1,\dots,a_m\}\subseteq \R^d$}\\\textnormal{such that $\Xset$ does {\em not} shatter $A$,}\\
\textnormal{has Lebesgue measure zero in $(\R^d)^m$}\end{array}\right\}\]
In other words, instead of searching for a {\em single} set $\mathcal{A}$ of size $m$ that is shattered by $\Xset$, we require
that {\em almost all} sets $\mathcal{A}$ of size $m$ are shattered by $\Xset$. It is trivial that $\vc(\Xset)\geq \vcae(\Xset)$, but in fact, 
the two may coincide---for example,\begin{itemize}
\item If $\Xset$ is the set of all $\ell_2$ balls in $\R^d$, then $\vcae(\Xset)=\vc(\Xset)=d+1$.
\item If $\Xset$ is the set of all half-spaces in $\R^d$, then $\vcae(\Xset)=\vc(\Xset)=d+1$.
\end{itemize}

In order to obtain a tight bound, we also need to define a slightly stronger notion of predictive coverage. Our previous definitions
(for marginal, conditional, and approximate conditional coverage) all calculated probabilities with respect to $P^{n+1}$ for some distribution $P$,
in other words, with the data points $(X_1,Y_1),\dots, (X_{n+1},Y_{n+1})$ drawn i.i.d.~from an arbitrary distribution. 
A more general setting is where these $n+1$ data points are instead assumed to be exchangeable (which includes i.i.d.~as a special case).
We thus define a notion of approximate conditional coverage under exchangeability, rather than the i.i.d.~assumption. We say that a procedure
\smash{$\Ch_n$} satisfies $(1-\alpha,\delta,\Xset)$-conditional coverage under exchangeability, denoted by $(1-\alpha,\delta,\Xset)$-CCE, if
\begin{multline}\label{eqn:approx_conditional_cov_Xset_exch}
\Ppst{\tilde{P}}{Y_{n+1}\in \Ch_n(X_{n+1})}{X_{n+1}\in\Xcal} \geq 1- \alpha\\ \textnormal{ for all exchangeable distributions $\tilde{P}$ on $(X_1,Y_1),\dots,(X_{n+1},Y_{n+1})$}\\\textnormal{ and all $\Xcal\in\Xset$ with $\Pp{\tilde{P}}{X_{n+1}\in\Xcal} \geq \delta$.}
\end{multline}
 Clearly, a procedure
\smash{$\Ch_n$} satisfying $(1-\alpha,\delta,\Xset)$-CCE will also satisfy $(1-\alpha,\delta,\Xset)$-CC by definition.
It is worth noting that all proofs of predictive coverage guarantees for conformal and split conformal prediction methods
do not require the i.i.d.~assumption but rather only need to assume exchangeability---that is, results such as Theorem~\ref{thm:Xset} continue
to hold, meaning that our split conformal method proposed in Section~\ref{sec:split_Xset} satisfies this stronger coverage property~\eqref{eqn:approx_conditional_cov_Xset_exch}.

We will now see how the VC dimension relates to the conditional coverage problem. We will show that:
\begin{itemize}
\item If $\vcae(\Xset)\geq 2n+2$, then the $(1-\alpha,\delta,\Xset)$-CCE property cannot be obtained beyond the trivial lower bound
given in Theorem~\ref{thm:main}.
\item On the other hand, if $\vc(\Xset)\ll \delta n/\log^2(n)$, then the split conformal method described in Section~\ref{sec:split_Xset}, which is 
guaranteed to satisfy $(1-\alpha,\delta,\Xset)$-CCE, produces prediction intervals of nearly optimal length under a location-family model.
\end{itemize}
An equivalent perspective is that with sufficiently many points $n$, the CCE property can be meaningfully attained.
We now formalize these results.

\subsubsection{A lower bound}
First, we will examine the setting where $\vcae(\Xset)\geq 2n+2$. In this setting, we will see that $(1-\alpha,\delta,\Xset)$-conditional
coverage (in its stronger form, with exchangeable rather than i.i.d.~data points) is incompatible with meaningful predictive intervals.
\begin{theorem}\label{thm:Xset_VC_lowerbd}
Suppose that \smash{$\Ch_n$} satisfies $(1-\alpha,\delta,\Xset)$-CCE as defined in~\eqref{eqn:approx_conditional_cov_Xset_exch}, where $\Xset$ satisfies $\vcae(\Xset)\geq 2n+2$. 
Then for all distributions $P$ where the marginal distribution $P_X$ is continuous with respect to Lebesgue measure,  we have
\[\EE{\leb(\Ch_n(X_{n+1}))} \geq \inf_{c\in[0,1]} \left\{\frac{1-\alpha}{1-c\alpha}\cdot L_P(1-c\alpha\delta)\right\}.\]
\end{theorem}
\noindent In other words, if $\vcae(\Xset)\geq 2n+2$, the lower bound proved here
is identical to that of Theorem~\ref{thm:main}, which is the trivial lower bound that can be obtained by 
simply requiring marginal coverage at a far stricter level.  (For example, if we take $\Xset$ to be the collection of 
all balls or all half-spaces in $\R^d$ for $d\geq 2n+1$, then this condition on $\vcae(\Xset)$ will hold.)
We remark that it is possible to prove a similar result for the $(1-\alpha,\delta,\Xset)$-CC condition (rather than the stronger $(1-\alpha,\delta,\Xset)$-CCE condition),
but in that case we are only able to show this result when $\vcae(\Xset)\gg n^2$.

\subsubsection{An upper bound}
Next, we prove that efficient prediction is possible when the VC dimension is low.
 
Since our construction given in~\eqref{eqn:PI_Xset} uses a symmetric interval
around an initial model $\muh_{n_0}$, with the width of the interval selected
to cover the worst-case scenario in terms of the choice of $\Xcal$, we can only hope
for efficiency as compared to the best ``oracle'' interval of this form.
For a fixed function $\mu:\R^d\rightarrow\R$ and for any $\Xcal\in\Xset$ with nonzero probability under $P_X$, define
\begin{multline*}q^*_{P,\mu,\alpha}(\Xcal) = \textnormal{the $(1-\alpha)$-quantile
of $|Y-\mu(X)|$,}\\\textnormal{under the distribution $(X,Y)\sim P$ conditional on $X\in\Xset$.}\end{multline*}
Next, for any $x\in\R^d$, define
\[q^*_{P,\mu,\alpha,\delta}(x) = \sup_{\Xcal\in\Xset : x\in\Xcal, P_X(\Xcal)\geq\delta} q^*_{P,\mu,\alpha}(\Xcal),\]
the maximum quantile over any set $\Xcal$ containing the point $x$.
We will then consider the ``oracle'' prediction interval
\begin{equation}\label{eqn:Xset_oracle}C_{P,\mu,\alpha,\delta}^*(x) = \big[\mu(x)-  q^*_{P,\muh_{n_0},\alpha,\delta}(x) ,\mu(x)+q^*_{P,\muh_{n_0},\alpha,\delta}(x) \big].\end{equation}
We can easily verify that $C_{P,\mu,\alpha,\delta}^*(x)$ satisfies
\[\Ppst{P}{Y\in C_{P,\mu,\alpha,\delta}^*(X)}{X\in\Xcal}\geq 1-\alpha\]
for all $\Xcal\in\Xset$ with $P_X(\Xcal)\geq\delta$.

Our main result proves that, if the collection $\Xset$ has sufficiently small VC dimension,
then with high probability the prediction interval $\Ch_n$ constructed in~\eqref{eqn:PI_Xset} above
is essentially the same as the ``oracle'' interval defined in~\eqref{eqn:Xset_oracle}, when constructed around
the pre-trained model $\mu= \muh_{n_0}$. To formalize this, we show that $\Ch_n$ is bounded above and below 
by oracle intervals with slightly perturbed values of $\alpha$ and $\delta$.
\begin{theorem}\label{thm:Xset_VC_upperbd}
Assume that $\vc(\Xset)\geq 1$ and $n_1\geq 2$. Then for every $x\in\R^d$,
if $\vc(\Xset)\leq c \cdot \frac{ \delta n_1}{\log^2(n_1)}$, then
the split conformal prediction interval \smash{$\Ch_n$} defined in~\eqref{eqn:PI_Xset} 
 satisfies 
\[\Pp{P^n}{C_{P,\muh_{n_0},\alpha_+,\delta_+}^*(x)\subseteq \Ch_n(X_{n+1})\subseteq  C_{P,\muh_{n_0},\alpha_-,\delta_-}^*(x)}\geq 1 - \frac{1}{n_1},\]
where
\[\alpha_+= \alpha + c_\alpha\sqrt{\frac{\vc(\Xset)\log^2(n_1)}{\delta n_1}}, \quad \alpha_- =  \alpha - c_\alpha\sqrt{\frac{\vc(\Xset)\log^2(n_1)}{\delta n_1}}\]
and
\[\delta_+ = \delta + c_\delta\sqrt{\frac{\vc(\Xset)\log^2(n_1)}{n_1}}, \quad \delta_- = \delta - c_\delta\sqrt{\frac{\vc(\Xset)\log^2(n_1)}{n_1}}\]
where $c,c_\alpha,c_\delta$ are universal constants.
\end{theorem}

\subsubsection{Special case: the location family with i.i.d.~noise}\label{sec:location_family}
While the result given in Theorem~\ref{thm:Xset_VC_upperbd}
 is quite general (we do not assume anything about the distribution $P$), we can consider
a special case where, given strong conditions on $P$, the prediction interval $\widehat{C}_n(X_{n+1})$ nearly
matches a much stronger oracle---namely, the narrowest possible valid prediction interval.

Our discussion for this setting will closely follow the work of \citet{lei2018distribution}, for the split conformal method.
We first describe their results. Their work assumes a
 location-family model:
\begin{equation}\label{eqn:loc_family}
\begin{array}{c}
\textnormal{The distribution of $Y |  X$ is given by $Y_i = \mu_P(X_i)+\eps_i$,}\\
\textnormal{where $\mu_P(x)$ is a fixed function, and the $\eps_i$'s are i.i.d.~with density $f_\eps$,}\\\textnormal{where $f_\eps(t)$ is
symmetric around $t=0$, and nonincreasing for $t\geq 0$.}\\
\end{array}
\end{equation}
 \citet{lei2018distribution} additionally assume that the estimator $\muh_{n_0}(x)$ of the true mean function $\mu_P(x)$ is consistent---Assumption
 A4 in their work requires that
\begin{equation}\label{eqn:muh_consistent}
\PP{\EEst{\left(\muh_{n_0}(X) - \mu_P(X)\right)^2}{\muh_{n_0}}\leq \eta_{n_0}} \geq 1-\rho_{n_0},
\end{equation}
where we should think of the quantities $\eta_{n_0},\rho_{n_0}$ as small or vanishing.
To interpret this assumption, the probability on the outside is taken with respect to the training data $(X_1,Y_1),\dots,(X_{n_0},Y_{n_0})$
used to fit the model $\muh_{n_0}$, while the conditional expectation on the inside is taken with respect to a new draw $X\sim P_X$.

Under conditions~\eqref{eqn:loc_family} and~\eqref{eqn:muh_consistent},
  \citet{lei2018distribution} prove that
the split conformal method~\eqref{eqn:split_conformal_define} is asymptotically efficient as $n_0,n_1\rightarrow\infty$, satisfying bounds of the form
\begin{equation}\label{eqn:split_conformal_asymp_efficient_sketch}
\leb\big(\Ch_n(X_{n+1})\,\triangle\,C_P^*(X_{n+1})\big) = o_P(1),
\end{equation}
where $\triangle$ denotes the symmetric set difference, and where
$C_P^*(x)$ denotes the ``oracle'' prediction interval that we would build if we knew the distribution $P$---under the simple
model~\eqref{eqn:loc_family} for $P$ above, this interval has the form
\[C_P^*(x) = \mu_P(x) \pm q^*_{\eps,\alpha},\]
where $q^*_{\eps,\alpha}$ denotes the $(1-\alpha/2)$ quantile of $f_\epsilon$ (i.e., the $(1-\alpha)$-quantile of the distribution of $|\eps|$).

We now extend this result to the setting of approximate conditional coverage.
Specifically, working under the same assumptions, we will prove that our proposed algorithm~\eqref{eqn:PI_Xset}, 
which is constructed to satisfy the $(1-\alpha,\delta,\Xset)$-CC property, will also return an interval that is asymptotically
equivalent to the oracle interval $C_P^*$ as long as $\vc(\Xset)$ is not too large.
\begin{corollary}\label{cor:Xset_VC_upperbd}
Under the conditions of Theorem~\ref{thm:Xset_VC_upperbd} together with assumptions~\eqref{eqn:loc_family}
and~\eqref{eqn:muh_consistent}, 
if $\vc(\Xset)\leq c \cdot \frac{ \delta n_1}{\log^2(n_1)}$,
then the split conformal prediction interval \smash{$\Ch_n$} defined in~\eqref{eqn:PI_Xset} satisfies
\[\leb\big(\Ch_n(X_{n+1})\,\triangle\,C_P^*(X_{n+1})\big) \leq c'\left(\frac{\eta_{n_0}^{1/3}}{\delta^{1/2}} + \frac{\eta_{n_0}^{1/3} + \sqrt{\frac{\vc(\Xset)\log^2(n_1)}{\delta n_1}} }{f_\eps(q^*_{\eps,\alpha/2})}\right)\]
with probability at least $1-\frac{1}{n_1} - 2\rho_{n_0} - \eta_{n_0}^{1/3}$, where $c,c'$ are universal constants.
\end{corollary}
\noindent In other words, for a location-family model with a consistent estimate of the true mean function ($\eta_{n_0},\rho_{n_0} \to 0$), the interval \smash{$\Ch_n$} defined in~\eqref{eqn:PI_Xset}
is able to satisfy restricted conditional coverage in the distribution-free setting, while matching the best possible ``oracle'' prediction interval length asymptotically as $n_0,n_1\to \infty$.

\section{Discussion}
In this work, we have explored the possible definitions of approximate conditional coverage for distribution-free predictive inference,
with the goal of finding meaningful definitions that are strong enough to achieve some of the practical benefits of conditional coverage (i.e., patients
feel assured that their personalized predictions have some level of accuracy), but weak enough to still allow for the possibility of meaningful
distribution-free procedures. We find that requiring $(1-\alpha,\delta)$-conditional coverage to hold, i.e., coverage at level $1-\alpha$ over every subgroup
with probability at least $\delta$ within the overall population, is too strong of a condition---our main result establishes a lower bound on the resulting prediction
interval length, and demonstrates that meaningful procedures cannot be constructed with this property. By relaxing the desired property
to $(1-\alpha,\delta,\Xset)$-conditional coverage,
i.e., coverage at level $1-\alpha$ over every subgroup $\Xcal\in\Xset$ that has probability at least $\delta$, we see that sufficiently restricting the class  $\Xset$
does allow for nontrivial prediction intervals. 

Many open questions remain after our preliminary findings. In particular, what types of classes
$\Xset$ are most meaningful for defining this restricted form of approximate conditional coverage? Furthermore, for nearly any class $\Xset$, computation for the split conformal method constructed
in Section~\ref{sec:split_Xset} may pose a serious challenge---how can we efficiently compute predictive intervals for this problem?

Another direction for relaxing $(1-\alpha,\delta)$-CC property is to require it to hold only over some distributions $P$ (rather than
restricting to a class $\Xset$ of sets that we condition on). Is it possible to ensure that conditional coverage at level $1-\alpha$ holds,
not at some uniform tolerance level $\delta$, but at an adaptive tolerance level $\delta(P)$ that is low for ``well-behaved'' distributions $P$
but may be as large as $1$ (i.e., only ensuring marginal coverage) for degenerate distributions $P$? We leave these questions for future work.

\subsection*{Acknowledgements}
The authors are grateful to the American Institute of Mathematics for supporting and hosting our collaboration.
R.F.B.~was partially supported by the National Science Foundation via grant DMS--1654076 and by an Alfred P.~Sloan fellowship. 
E.J.C.~was partially supported by the Office of Naval Research under grant N00014-16-1-2712, by the National Science Foundation via grant DMS--1712800, and by a generous gift from TwoSigma.
R.F.B.~thanks Chao Gao, Samir Khan, and Haoyang Liu for helpful suggestions on an early version of this work.

\appendix
\section{Proof of main impossibility result (Theorem~\ref{thm:main})}
\subsection{A preliminary lemma}

In order to prove our main theorem, we rely on a key lemma:
\begin{lemma}\label{lem:X_to_XY}
Suppose that \smash{$\Ch_n$} satisfies $(1-\alpha,\delta)$-CC as defined in~\eqref{eqn:approx_conditional_cov}. Then for all distributions $P$ where the marginal distribution $P_X$ has no atoms, and for all measurable sets $\Bcal\subseteq \R^d\times\R$ with $P(\Bcal)\geq \delta$, we have
\[\PPst{Y_{n+1}\in\Ch_n(X_{n+1})}{(X_{n+1},Y_{n+1})\in\Bcal} \geq  1-\alpha.\]
\end{lemma}
\noindent Comparing this lemma to the definition of $(1-\alpha,\delta)$-CC, we see that the definition of approximate conditional coverage requires that the result of the lemma must hold for any set of the form $\Bcal = \Xcal\times\R$, i.e., conditioning on an event $X_{n+1}\in\Xcal$ (with probability at least $\delta$). The lemma extends the property to condition also on events that are defined jointly in $(X,Y)$. 

While this may initially appear to be a simple extension of the definition of $(1-\alpha,\delta)$-CC, the proof is not trivial, and the implications of this result are very significant. To see why, suppose that we construct $\Bcal$ to consist only of points $(x,y)$ such that $Y_{n+1}=y$ is in the extreme tail of its conditional distribution given $X_{n+1}=x$---specifically, outside the range given by the $\delta/2$ and $1-\delta/2$ conditional quantiles (so that the overall probability of $\Bcal$ is large enough, i.e., $\geq \delta$). The lemma claims that, even when $(X_{n+1},Y_{n+1})$ lands in this set, i.e., $Y_{n+1}$ is in the extreme tails of its conditional distribution given $X_{n+1}$, this value $Y_{n+1}$ is still quite likely to lie in \smash{$\Ch_n(X_{n+1})$}. This implies that \smash{$\Ch_n(X_{n+1})$} must indeed be very wide.

We will next formalize this intuition to prove our theorem.

\subsection{Proof of Theorem~\ref{thm:main}}

First, for each $x\in\R^d$ and each $s\in[0,1]$, define
\[C_{P,s}(x) = \left\{y : \PP{y\in\Ch_n(x)}> s\right\},\]
where the probability is taken with respect to the training data. Note that $C_{P,s}(x)$ is fixed, since it is defined as a function of the {\em distribution} of \smash{$\Ch_n(x)$}, not of the random interval \smash{$\Ch_n(x)$} itself. 

Next, for any fixed $x$, in expectation over the training data we have
\[\EE{\leb\big(\Ch_n(x)\big)} = \EE{\int_{y\in\R}\One{y\in\Ch_n(x)}\;\mathsf{d}y} = \int_{y\in\R}\PP{y\in\Ch_n(x)}\;\mathsf{d}y,\]
by Fubini's theorem. Now, we can rewrite
\[\PP{y\in\Ch_n(x)} = \int_{s=0}^1 \One{\PP{y\in\Ch_n(x)} > s}\;\mathsf{d}s =  \int_{s=0}^1\One{y\in C_{P,s}(x)}\;\mathsf{d}s,\]
and so plugging this in and applying Fubini's theorem again,
\[
\EE{\leb\big(\Ch_n(x)\big)}
= \int_{s=0}^1 \int_{y\in\R}\One{y\in C_{P,s}(x)}\;\mathsf{d}y\;\mathsf{d}s\\
=\int_{s=0}^1\leb\big(C_{P,s}(x)\big)\;\mathsf{d}s.
\]

Next, plugging in the test point $X_{n+1}$, and applying Fubini's theorem an additional time,
\begin{multline}\label{eqn:thm_step1}
\EE{\leb\big(\Ch_n(X_{n+1})\big)}
=\EE{\EEst{\leb\big(\Ch_n(X_{n+1})\big)}{X_{n+1}}}
=\EE{\int_{s=0}^1\leb\big(C_{P,s}(X_{n+1})\big)\;\mathsf{d}s}\\
=\int_{s=0}^1\EE{\leb\big(C_{P,s}(X_{n+1})\big)}\;\mathsf{d}s=\int_{s=0}^1\Ep{P_X}{\leb\big(C_{P,s}(X)\big)}\;\mathsf{d}s,
\end{multline}
where the last step holds since marginally $X_{n+1}\sim P_X$.

Next we define
\[\alpha_s = \Pp{P}{Y\not\in C_{P,s}(X)},\]
the marginal miscoverage rate  of the sets $C_{P,s}(x)$ (that is, we think of $C_{P,s}(x)$ as a deterministic prediction interval).
Then 
\begin{equation}\label{eqn:alpha_s_L_P}\Ep{P_X}{\leb\big(C_{P,s}(X)\big)}\geq L_P(1-\alpha_s)\end{equation}
 by the definition  of the minimal prediction interval length $L_P$ given in~\eqref{eqn:def_L_P}.
Since $s\mapsto \alpha_s$ is nondecreasing and right-continuous, and satisfies $\alpha_1 = 1$, we can define
\[s_\star = \min\{s\in[0,1] : \alpha_s \geq\delta\}.\]
Define also
\[\Bcal_+ = \left\{(x,y) : \PP{y\in\Ch_n(x)} \leq s_\star\right\}\textnormal{ and }
\Bcal_- = \left\{(x,y) : \PP{y\in\Ch_n(x)} < s_\star\right\}.\]
Then
\[\Pp{P}{(X,Y)\in\Bcal_+} =  \alpha_{s_\star} \geq \delta\textnormal{ and }
\Pp{P}{(X,Y)\in\Bcal_-}=\sup_{s<s_\star}\alpha_s\leq \delta.\]
Now, since $P$ is assumed to have no atoms (inheriting this property from the marginal $P_X$),
 by \citet[Proposition A.1]{dudley2011concrete} we can find a  measurable
 set $\Bcal$ such that
\[\Bcal_- \subseteq \Bcal \subseteq \Bcal_+\textnormal{ and }\Pp{P}{(X,Y)\in\Bcal} = \delta.\]
By definition of $\Bcal$, we have 
\begin{equation}\label{eqn:Bcal_implies}
\begin{array}{l}
(x,y)\in\Bcal \ \Rightarrow \ \PP{y\in\Ch_n(x)} \leq  s_\star, \\
(x,y)\not\in\Bcal \ \Rightarrow \ \PP{y\in\Ch_n(x)} \geq s_\star .\end{array}\end{equation}
Next, we can calculate
\begin{align*}
&\int_{s=0}^{s_\star} \alpha_s\;\mathsf{d}s 
= s_\star - \int_{s=0}^{s_\star} (1-\alpha_s)\;\mathsf{d}s\\
&=s_\star - \int_{s=0}^{s_\star} \Pp{P}{Y\in C_{P,s}(X)}\;\mathsf{d}s \\
&=s_\star - \int_{s=0}^{s_\star} \Pp{P}{\PPst{Y\in\Ch_n(X)}{X,Y}>s}\;\mathsf{d}s \\
&=s_\star - \int_{s=0}^1 \Pp{P}{\PPst{Y\in\Ch_n(X)}{X,Y} \wedge s_\star >s}\;\mathsf{d}s \\
&=s_\star - \Ep{P}{\PPst{Y\in\Ch_n(X)}{X,Y} \wedge s_\star} \\
&=s_\star - \left(\Ep{P}{\PPst{Y\in\Ch_n(X)}{X,Y} \cdot\One{(X,Y)\in\Bcal}} + \Ep{P}{s_\star\cdot\One{(X,Y)\not\in\Bcal}} \right)\\
&=s_\star - \PP{Y_{n+1}\in\Ch_n(X_{n+1}),(X_{n+1},Y_{n+1})\in\Bcal}  - s_\star \Pp{P}{(X,Y)\not\in\Bcal} \\
&=\delta \left(s_\star - \PPst{Y_{n+1}\in\Ch_n(X_{n+1})}{(X_{n+1},Y_{n+1})\in\Bcal} \right),
\end{align*}
where the last step holds since $\Pp{P}{(X,Y)\in\Bcal}=\PP{(X_{n+1},Y_{n+1})\in\Bcal}=\delta$ by construction. Next, by applying Lemma~\ref{lem:X_to_XY} to the set $\Bcal$, we have
\[ \PPst{Y_{n+1}\in\Ch_n(X_{n+1})}{(X_{n+1},Y_{n+1})\in\Bcal} \geq 1-\alpha\]
and therefore
\begin{equation}\label{eqn:thm_step2}\int_{s=0}^{s_\star} \alpha_s\;\mathsf{d}s  \leq \delta \left(s_\star - (1-\alpha) \right).\end{equation}
In particular, since the left-hand side is nonnegative, this proves that we must have $s_\star \geq 1- \alpha>0$ (we can assume that $\alpha<1$ since otherwise the theorem holds trivially).

Now, returning to~\eqref{eqn:thm_step1} and~\eqref{eqn:alpha_s_L_P}, we have
\begin{multline}\label{eqn:thm_step3}
\EE{\leb\big(\Ch_n(X_{n+1})\big)}
\geq \int_{s=0}^1 L_P(1-\alpha_s)\;\mathsf{d}s
\geq \int_{s=0}^{s_\star} L_P(1-\alpha_s)\;\mathsf{d}s\\
= s_\star \int_{s=0}^{s_\star} \frac{1}{s_\star} L_P(1-\alpha_s)\;\mathsf{d}s
\geq s_\star L_P\left(1-\int_{s=0}^{s_\star} \frac{1}{s_\star} \alpha_s\;\mathsf{d}s\right),
\end{multline}
where the last step uses Jensen's inequality, together with the fact that $\alpha\mapsto L_P(1-\alpha)$ is convex.
(To verify this,
let $C_P\in \Ccal_P(1-\alpha)$ and $C_P'\in\Ccal_P(1-\alpha')$, and then define $C_P''(x)$ as the random interval that outputs
$C_P(x)$ with probability $(1-t)$ and $C_P'(x)$ with probability $t$. Then it is easy to verify that $C_P''\in \Ccal_P(1-\alpha'')$ where $\alpha''=(1-t)\alpha+ t\alpha'$,
and that $\Ep{P_X}{\leb(C_P''(X))} = (1-t)\Ep{P_X}{\leb(C_P(X))}+t\Ep{P_X}{\leb(C_P''(X))}$. This is sufficient to establish convexity.)

Combining~\eqref{eqn:thm_step2} and~\eqref{eqn:thm_step3}, we obtain
\[\EE{\leb\big(\Ch_n(X_{n+1})\big)} \geq s_\star L_P\left(1-\delta \left(1 - \frac{1-\alpha}{s_\star}\right)\right),\]
since $L_P$ is nondecreasing. Finally, define
\[c = \frac{1}{\alpha} - \frac{1-\alpha}{s_\star\alpha}.\]
Since we have verified that $1-\alpha \leq s_\star \leq 1$, this means that $c\in[0,1]$, and
plugging in this choice of $c$, we obtain
\[\EE{\leb\big(\Ch_n(X_{n+1})\big)} \geq   \frac{1-\alpha}{1-c\alpha}L_P(1-c\alpha\delta),\]
which proves the theorem.
\subsection{Proof of Lemma~\ref{lem:X_to_XY}} 
Let $\delta' = \Pp{P}{(X,Y)\in\Bcal}\geq\delta$. We will assume that $\delta'<1$ (since the case $\delta'=1$ is trivial).
Fix a large integer $M\geq n+1$. First, draw $M$ data points $(X_0^{(1)},Y_0^{(1)}),\dots,(X_0^{(M)},Y_0^{(M)})$  i.i.d.~from $(X,Y)\sim P$ conditional on $(X,Y)\not\in\Bcal$, and $M$ additional data points $(X_1^{(1)},Y_1^{(1)}),\dots,(X_1^{(M)},Y_1^{(M)})$ i.i.d.~from $(X,Y)\sim P$ conditional on $(X,Y)\in\Bcal$. Let $\Lcal$ denote this draw of the $2M$ data points. Since $P_X$ has no atoms, with probability $1$ all the $X_0^{(i)}$'s and $X_1^{(i)}$'s are distinct, so from this point on we assume that this is true.

Next suppose that we draw indices $m_1,\dots,m_{n+1}$ without replacement from the set $\{1,\dots,M\}$. Independently for each $i=1,\dots,n+1$, set
\begin{equation}\label{eqn:Lcal_sample}(X_i,Y_i) = \begin{cases}(X_0^{(m_i)},Y_0^{(m_i)}),&\textnormal{ with probability $1-\delta'$,}\\(X_1^{(m_i)},Y_1^{(m_i)}),&\textnormal{ with probability $\delta'$.}\end{cases}\end{equation}
We can clearly see that, after marginalizing over $\Lcal$, this is equivalent to drawing the data points $(X_i,Y_i)$ i.i.d.~from $P$. Therefore, we have
\begin{multline*}\PP{Y_{n+1}\not\in\Ch_n(X_{n+1}),(X_{n+1},Y_{n+1})\in\Bcal} \\
=\EE{\PPst{Y_{n+1}\not\in\Ch_n(X_{n+1}),(X_{n+1},Y_{n+1})\in\Bcal}{\Lcal}},\end{multline*}
where, on the right-hand side, after conditioning on $\Lcal$, the data points $(X_i,Y_i)$ are drawn according to~\eqref{eqn:Lcal_sample}.

Next consider an alternate distribution where we  draw the $n+1$ data points $(X_i,Y_i)$ from $\Lcal$ but now drawing {\em with} replacement. Specifically, fixing $\Lcal$, let $Q(\Lcal)$ be the 
discrete distribution that places probability $\frac{1-\delta'}{M}$ on each point $(X_0^{(m)},Y_0^{(m)})$, and probability $\frac{\delta'}{M}$ on each point $(X_1^{(m)},Y_1^{(m)})$, for $m=1,\dots,M$.
The product distribution $\big(Q(\Lcal)\big)^{n+1}$ is therefore equivalent to sampling indices $m_1,\dots,m_{n+1}$ {\em with} replacement from the set $\{1,\dots,M\}$, and then defining $(X_i,Y_i)$ again according to~\eqref{eqn:Lcal_sample}.

Now, if $M$ is very large relative to $n$, it is extremely unlikely that we would have $m_i=m_{i'}$ for any $i\neq i'$, when drawing from $\big(Q(\Lcal)\big)^{n+1}$. Specifically, we can easily check that this probability is bounded by $\frac{n^2}{M}$, and so for any fixed $\Lcal$, the total variation distance between the distribution
given in~\eqref{eqn:Lcal_sample} (i.e., sampling without replacement) and the distribution $(Q(\Lcal))^{n+1}$ (i.e., sampling with replacement)
is bounded by $\frac{n^2}{M}$. Therefore, 
\begin{multline*}
\PPst{Y_{n+1}\not\in\Ch_n(X_{n+1}),(X_{n+1},Y_{n+1})\in\Bcal}{\Lcal}\\
\leq \Pp{(Q(\Lcal))^{n+1}}{Y_{n+1}\not\in\Ch_n(X_{n+1}),(X_{n+1},Y_{n+1})\in\Bcal}+ \frac{n^2}{M},\end{multline*}
where on the left-hand side, after conditioning on $\Lcal$, the data points $(X_i,Y_i)$ are drawn according to~\eqref{eqn:Lcal_sample}.

Next, for any $\Lcal$, define the set
\[\Xcal(\Lcal) = \{X_1^{(1)},\dots,X_1^{(M)}\}.\]
Note that, for $(X,Y)\sim Q(\Lcal)$, by construction we have $X\in\Xcal(\Lcal)$ if and only if $(X,Y)\in\Bcal$ (since we have assumed that $\Lcal$ is chosen
so that $X_0^{(1)},\dots,X_0^{(M)},X_1^{(1)},\dots,X_1^{(M)}$ are all distinct), and
\[\Pp{Q(\Lcal)}{X\in\Xcal(\Lcal)} =\Pp{Q(\Lcal)}{(X,Y)\in\Bcal} = \delta'\geq \delta.\]
Therefore, since \smash{$\Ch_n$} satisfies $(1-\alpha,\delta)$-CC with respect to any distribution, we must have
\begin{multline*}\Pp{(Q(\Lcal))^{n+1}}{Y_{n+1}\not\in\Ch_n(X_{n+1}),(X_{n+1},Y_{n+1})\in\Bcal} \\= \Pp{(Q(\Lcal))^{n+1}}{Y_{n+1}\not\in\Ch_n(X_{n+1}),X_{n+1}\in\Xcal(\Lcal)}\\
= \Ppst{(Q(\Lcal))^{n+1}}{Y_{n+1}\not\in\Ch_n(X_{n+1})}{X_{n+1}\in\Xcal(\Lcal)} \cdot \delta'\leq \alpha\delta'\end{multline*}
 for every fixed $\Lcal$ where $X_0^{(1)},\dots,X_0^{(M)},X_1^{(1)},\dots,X_1^{(M)}$ are distinct.
 Combining everything, therefore,
\begin{multline*}\PP{Y_{n+1}\not\in\Ch_n(X_{n+1}),(X_{n+1},Y_{n+1})\in\Bcal}\\
\leq \EE{\Pp{(Q(\Lcal))^{n+1}}{Y_{n+1}\not\in\Ch_n(X_{n+1}),(X_{n+1},Y_{n+1})\in\Bcal} + \frac{n^2}{M}}
\leq \alpha\delta' + \frac{n^2}{M},
\end{multline*}
where the expectation is taken with respect to the random draw of $\Lcal$.
Since $M$ can be taken to be arbitrarily large, we therefore have
\[\PP{Y_{n+1}\not\in\Ch_n(X_{n+1}),(X_{n+1},Y_{n+1})\in\Bcal} \leq \alpha\delta' = \alpha\cdot \PP{(X_{n+1},Y_{n+1})\in\Bcal},\]
which concludes the proof of the lemma.

\section{Additional proofs}

\subsection{Proof of Lemma~\ref{lem:alphadelta_MC_extend}}\label{app:proof_lem:alphadelta_MC_extend}
Let $A\sim\textnormal{Bernoulli}\left(\frac{1-\alpha}{1-c\alpha}\right)$ be the Bernoulli variable indicating whether $\Ch'_n(x)$ is defined as \smash{$\Ch_n(x)$} (if $A=1$) or as the empty set (if $A=0$). Then, for any $\Xcal$ with $P_X(\Xcal)\geq \delta$, we have
\begin{multline*}\PPst{Y_{n+1}\in \Ch'_n(X_{n+1})}{X_{n+1}\in\Xcal} = \PPst{A=1,Y_{n+1}\in \Ch_n(X_{n+1})}{X_{n+1}\in\Xcal}\\=\frac{1-\alpha}{1-c\alpha}\cdot \PPst{Y_{n+1}\in \Ch_n(X_{n+1})}{X_{n+1}\in\Xcal}\geq \frac{1-\alpha}{1-c\alpha}\cdot(1-c\alpha) = 1-\alpha,\end{multline*}
where the inequality holds since \smash{$\Ch_n$} satisfies $(1-c\alpha,\delta)$-CC by Lemma~\ref{lem:alphadelta_MC}.

\subsection{Proof of Theorem~\ref{thm:Xset}}
Fix any distribution $P$ and any $\Xcal\in\Xset$ with $P_X(\Xcal) \geq \delta$. Let
\[R_{n+1} = \big|Y_{n+1}-\muh_{n_0}(X_{n+1})\big|\]
be the residual of the test point. By definition of the procedure, we can see that
\begin{multline*}
\PPst{Y_{n+1}\not\in\Ch_n(X_{n+1})}{X_{n+1}\in\Xcal} = \PPst{R_{n+1}>\qh_{n_1}(X_{n+1})}{X_{n+1}\in\Xcal} \\
\leq \PPst{\Xcal\not\in\widehat{\Xset}_{n_1}}{X_{n+1}\in\Xcal} + \PPst{R_{n+1}>\qh_{n_1}(\Xcal)}{X_{n+1}\in\Xcal}.\end{multline*}
The first probability depends only on the held-out portion of the training data, i.e., data points $i=n_0+1,\dots,n$. We have
\[\Xcal\not\in\widehat{\Xset}_{n_1} \ \Rightarrow \ \sum_{i=n_0+1}^n \One{X_i\in\Xcal} <  \delta n_1\left(1  - \sqrt{\frac{2\log(n_1)}{\delta n_1}}\right).\]
Since each $X_i$ has probability at least $\delta$ of lying in $\Xcal$, therefore this probability is bounded by 
\[\PP{\textnormal{Binomial}(n_1,\delta) <    \delta n_1\left(1  - \sqrt{\frac{2\log(n_1)}{\delta n_1}}\right)} \leq  \frac{1}{n_1},\]
where the inequality holds by the multiplicative Chernoff bound. Therefore, what we have so far is
\[\PPst{Y_{n+1}\not\in\Ch_n(X_{n+1})}{X_{n+1}\in\Xcal}  \leq \frac{1}{n_1} + \PPst{R_{n+1}>\qh_{n_1}(\Xcal)}{X_{n+1}\in\Xcal}.\]
Next let $I=\{i:n_0+1\leq i\leq n, X_i\in\Xcal\}$. Then $|I|=\widehat{N}_{n_1}(\Xcal)$, and by definition of $\qh_{n_1}(\Xcal)$, we see that $R_{n+1}>\qh_{n_1}(\Xcal)$ if and only if $R_{n+1}$ is not one of the $\left\lceil\left(1-\alpha+\frac{1}{n_1}\right)\cdot (|I|+1)\right\rceil$ smallest
values of $\{R_i: i\in I\cup\{n+1\}\}$. Now, after conditioning on $I$ and on the event $X_{n+1}\in\Xcal$, by distribution of the data we see that these residuals are exchangeable.
Therefore this event has probability exactly 
\[1 - \frac{\left\lceil\left(1-\alpha+\frac{1}{n_1}\right)\cdot (|I|+1)\right\rceil}{|I|+1} \leq \alpha -\frac{1}{n_1} \]
 after conditioning on $I$ and on the event that $X_{n+1}\in\Xcal$. This bound is therefore true also after marginalizing over $I$, and so
$\PPst{R_{n+1}>\qh_{n_1}(\Xcal)}{X_{n+1}\in\Xcal}\leq \alpha - \frac{1}{n_1}$, which concludes the proof.

\subsection{Proof of Theorem~\ref{thm:Xset_VC_lowerbd}}

First, we need to show that Lemma~\ref{lem:X_to_XY} holds in this setting. 
\begin{lemma}\label{lem:X_to_XY_Xset}
Suppose that \smash{$\Ch_n$} satisfies $(1-\alpha,\delta,\Xset)$-CCE as defined in~\eqref{eqn:approx_conditional_cov_Xset_exch},
where $\Xset$ satisfies $\vcae(\Xset)\geq 2n+2$. Then for all distributions $P$ where the marginal distribution $P_X$ is continuous with respect to Lebesgue measure, for all $\Bcal\subseteq \R^d\times\R$ with $\Pp{P}{(X,Y)\in\Bcal}\geq \delta$,
\[\PPst{Y_{n+1}\in\Ch_n(X_{n+1})}{(X_{n+1},Y_{n+1})\in\Bcal} \geq1- \alpha.\]
\end{lemma}
With this lemma in place, the proof of Theorem~\ref{thm:Xset_VC_lowerbd} follows exactly as the proof of our initial result, Theorem~\ref{thm:main}.
We now turn to proving the lemma.

\begin{proof}[Proof of Lemma~\ref{lem:X_to_XY_Xset}]
The proof of this lemma is similar to that of Lemma~\ref{lem:X_to_XY}, except that instead of taking $M$ samples from $\Bcal$ and from $\Bcal^c$
for an arbitrarily large integer $M$, we only need to take $n+1$ from each set.

Let $\delta' = \Pp{P}{(X,Y)\in\Bcal}\geq\delta$.  We can assume that $\delta'<1$ (otherwise, the bound claimed in the lemma is trivial).
Draw $n+1$ data points $(X_0^{(1)},Y_0^{(1)}),\dots,(X_0^{(n+1)},Y_0^{(n+1)})$  i.i.d.~from $(X,Y)\sim P$ conditional on $(X,Y)\not\in\Bcal$, and $n+1$ additional data points $(X_1^{(1)},Y_1^{(1)}),\dots,(X_1^{(n+1)},Y_1^{(n+1)})$ i.i.d.~from $(X,Y)\sim P$ conditional on $(X,Y)\in\Bcal$. Let $\Lcal$ denote this draw of the $2n+2$ data points. 

Next, we draw a permutation $\pi$ of the set $\{1,\dots,n+1\}$ uniformly at random, and draw $B_1,\dots,B_{n+1}\iidsim\textnormal{Bernoulli}(\delta')$
independently of all other random variables.
Define
\[(X_i,Y_i) =  \begin{cases}(X_0^{(\pi_i)},Y_0^{(\pi_i)}),&\textnormal{ if $B_i=0$,}\\(X_1^{(\pi_i)},Y_1^{(\pi_i)}),&\textnormal{ if $B_i=1$.}\end{cases}.\]
We can clearly see that, after marginalizing over $\Lcal$, this is equivalent to drawing the data points $(X_i,Y_i)$ i.i.d.~from $P$. Therefore, we have
\begin{multline}\label{eqn:Lcal_tower}\PP{Y_{n+1}\not\in\Ch_n(X_{n+1}),(X_{n+1},Y_{n+1})\in\Bcal} \\
=\EE{\PPst{Y_{n+1}\not\in\Ch_n(X_{n+1}),(X_{n+1},Y_{n+1})\in\Bcal}{\Lcal}},\end{multline}
where, on the right-hand side, after conditioning on $\Lcal$, the data points $(X_i,Y_i)$ are defined by the permutation $\pi$ and the Bernoulli variables 
$B_1,\dots,B_{n+1}$. 

Next consider the distribution of the data conditional on $\Lcal$, which we denote by $\tilde{P}(\Lcal)$. 
Since the permutation $\pi$ is drawn uniformly at random, and the $B_i$'s are i.i.d.,
 it is clear that the $n+1$ data points $(X_1,Y_1),\dots,(X_{n+1},Y_{n+1})$ are exchangeable under the distribution $\tilde{P}(\Lcal)$.
 Therefore for any fixed $\Lcal$ and for any set $\Xcal\in\Xset$ with $\Pp{\tilde{P}(\Lcal)}{X_{n+1}\in\Xcal}\geq \delta$,
 the $(1-\alpha,\delta,\Xset)$-CCE property ensures that
\[ \Ppst{\tilde{P}(\Lcal)}{Y_{n+1}\in\Ch_n(X_{n+1})}{X_{n+1}\in\Xcal}\geq 1-\alpha.\]

Now, fixing $\Lcal$, define the set $\Xcal(\Lcal)$ to be any element of $\Xset$ such that
\[\Xcal(\Lcal) \ni X_1^{(1)},\dots,X_1^{(n+1)}, \quad \Xcal(\Lcal) \not\ni X_0^{(1)},\dots,X_0^{(n+1)}.\]
(Since we have assumed that $\vcae(\Xset)\geq 2n+2$, 
 and that $P_X$ is continuous with respect to Lebesgue measure, such a set $\Xcal(\Lcal)\in\Xset$ exists
with probability one for any random draw of $\Lcal$.)
Note that, under the distribution $\tilde{P}(\Lcal)$, we have $X_{n+1}\in\Xcal(\Lcal)$ if and only if
$(X_{n+1},Y_{n+1})\in\Bcal$, and
\[\Pp{\tilde{P}(\Lcal)}{(X_{n+1},Y_{n+1})\in\Bcal} = \Pp{\tilde{P}(\Lcal)}{X_{n+1}\in\Xcal(\Lcal)} = \PP{B_{n+1}=1} = \delta'\geq \delta.\]
Returning to the above, we therefore have
\begin{multline*} \Pp{\tilde{P}(\Lcal)}{Y_{n+1}\in\Ch_n(X_{n+1}),(X_{n+1},Y_{n+1})\in\Bcal}\\
= \Pp{\tilde{P}(\Lcal)}{Y_{n+1}\in\Ch_n(X_{n+1}),X_{n+1}\in\Xcal(\Lcal)}\\ = \Ppst{\tilde{P}(\Lcal)}{Y_{n+1}\in\Ch_n(X_{n+1})}{X_{n+1}\in\Xcal(\Lcal)} \cdot \delta' \geq (1-\alpha)\cdot\delta'.\end{multline*}
Then, returning to~\eqref{eqn:Lcal_tower}, 
\begin{multline*}\PP{Y_{n+1}\in\Ch_n(X_{n+1}),(X_{n+1},Y_{n+1})\in\Bcal} \\
=\EE{\PPst{Y_{n+1}\in\Ch_n(X_{n+1}),(X_{n+1},Y_{n+1})\in\Bcal}{\Lcal}}\\
=\EE{\Pp{\tilde{P}(\Lcal)}{Y_{n+1}\in\Ch_n(X_{n+1}),(X_{n+1},Y_{n+1})\in\Bcal}}\\ \geq \EE{(1-\alpha)\cdot \delta'}=(1-\alpha)\cdot \delta'.
\end{multline*}
Therefore,
\[\PPst{Y_{n+1}\in\Ch_n(X_{n+1})}{(X_{n+1},Y_{n+1})\in\Bcal} \geq \frac{(1-\alpha)\cdot \delta'}{\delta'} = 1-\alpha,\]
which proves the lemma.
\end{proof}

\subsection{Proof of Theorem~\ref{thm:Xset_VC_upperbd}}
Let $\mu = \muh_{n_0}$. Throughout this proof, we will condition on the data $(X_1,Y_1),\dots,(X_{n_0},Y_{n_0})$,
and will therefore treat this model
as fixed---the probability bound will hold with respect to the distribution of the $n_1$ holdout points (and therefore, the bound also 
holds after marginalizing over the initial $n_0$ training points).

We will first see that it is sufficient to prove that, with high probability, the following two bounds hold:
\begin{equation}\label{eqn:concentration_Xset}\Xset_{x,+}\subseteq \widehat{\Xset}_{n_1}\subseteq \Xset_{x,-},\end{equation}where
we define $\Xset_{x,+} = \{\Xcal\in\Xset:x\in\Xcal, \, P_X(\Xcal)\geq \delta_+\}$ and $\Xset_{x,-} = \{\Xcal\in\Xset:x\in\Xcal, \, P_X(\Xcal)\geq \delta_-\}$,
and
\begin{equation}\label{eqn:concentration_quantile}q^*_{P,\mu,\alpha_+}(\Xcal) \leq \qh_{n_1}(\Xcal) \leq q^*_{P,\mu,\alpha_-}(\Xcal)\textnormal{\quad for all $\Xcal\in\Xset_{x,-}$}.\end{equation}
If these two statements hold, then we have
\[q^*_{P,\mu,\alpha_+,\delta_+}(x) = \sup_{\Xcal\in\Xset_{x,+}}q^*_{P,\mu,\alpha_+}(\Xcal)\leq \sup_{\Xcal\in\widehat\Xset_{n_1}}q^*_{P,\mu,\alpha_+}(\Xcal)
\leq \sup_{\Xcal\in\widehat\Xset_{n_1}}\qh_{n_1}(\Xcal) = \qh_{n_1}(x),\]
and similarly
\[q^*_{P,\mu,\alpha_-,\delta_-}(x) = \sup_{\Xcal\in\Xset_{x,-}}q^*_{P,\mu,\alpha_-}(\Xcal)\geq \sup_{\Xcal\in\widehat\Xset_{n_1}}q^*_{P,\mu,\alpha_-}(\Xcal)
\geq \sup_{\Xcal\in\widehat\Xset_{n_1}}\qh_{n_1}(\Xcal) = \qh_{n_1}(x).\]By construction of the intervals, we therefore
see that
$C_{P,\mu,\alpha_+,\delta_+}^*(x)\subseteq \Ch_n(x)\subseteq  C_{P,\mu,\alpha_-,\delta_-}^*(x)$, which is the claim
in the theorem.

Now we verify that~\eqref{eqn:concentration_Xset} and~\eqref{eqn:concentration_quantile} both hold with high probability.
First, by \citet[Section 2.3 (Bousquet bound) + Theorem 3.9]{koltchinskii2011oracle},
we can verify the following concentration result:\footnote{To obtain this bound, we need to apply \citet[Theorems 2.5+3.9]{koltchinskii2011oracle} $\O{\log(n_1)}$ many times, once
for each class $\Xset_j = \{\Xcal\in\Xset: P_X(\Xcal)\leq 2^{-j}\}$, for $j=1,\dots,\O{\log(n_1)}$ (i.e., a peeling argument).}
\begin{equation}\label{eqn:concentration_1}\PP{\left|\frac{\widehat{N}_{n_1}(\Xcal)}{n_1} -P_X(\Xcal)\right| \leq \Delta_{\textnormal{conc}}(\Xcal)\textnormal{ for all $\Xcal\in\Xset$}} \geq 1- \frac{1}{3n_1},\end{equation}
where 
\[\Delta_{\textnormal{conc}}(\Xcal) =  c\sqrt{P_X(\Xcal)}\cdot \sqrt{\frac{\vc(\Xset)\log^2(n_1)}{n_1}} + \frac{c\log(n_1)}{n_1},\]
for a universal constant $c$.

Next, for any $\Xcal\in\Xset$, define
\[\tilde\Xcal= \left\{(x,y)\in\R^d\times\R : x\in\Xcal\textnormal{ and }|y - \mu(x)|> q^*_{P,\mu,\alpha_-}(\Xcal)\right\}.\]
Lemma~\ref{lem:vc} below will verify that
\[\vc\Big(\{\tilde\Xcal : \Xcal\in\Xset\}\Big)\leq \vc(\Xset)+1.\]
Therefore, again applying 
\citet[Bousquet bound (Section 2.3) + Theorem 3.9]{koltchinskii2011oracle} as above, if the universal constant $c$ is chosen appropriately then it holds that
\begin{multline*}
\mathbb{P}\Bigg\{\left|\frac{1}{n_1}\sum_{i=n_0+1}^n \One{(X_i,Y_i)\in\tilde\Xcal} - \Pp{P}{(X,Y)\in\tilde\Xcal}\right|\\ \leq\Delta_{\textnormal{conc}}(\Xcal)
\quad\forall\ \Xcal\in\Xset_{x,-}\Bigg\}\geq 1 - \frac{1}{3n_1}.
\end{multline*}
Plugging in the definition of $\tilde\Xcal$ and of $q^*_{P,\mu,\alpha_-}(\Xcal)$, this means that
 \begin{multline}\label{eqn:concentration_2}
\mathbb{P}\Bigg\{\frac{1}{n_1}\sum_{i=n_0+1}^n \One{X_i\in\Xcal, \ |Y_i - \mu(X_i)|>q^*_{P,\mu,\alpha_-}(\Xcal)}\\ \leq  \alpha_-P_X(\Xcal) + \Delta_{\textnormal{conc}}(\Xcal)
\quad\forall\ \Xcal\in\Xset_{x,-}\Bigg\}\geq 1 - \frac{1}{3n_1}.
\end{multline}
An analogous argument can be used to prove that
 \begin{multline}\label{eqn:concentration_3}
\mathbb{P}\Bigg\{\frac{1}{n_1}\sum_{i=n_0+1}^n \One{X_i\in\Xcal, \ |Y_i - \mu(X_i)|\geq q^*_{P,\mu,\alpha_+}(\Xcal)}\\ \geq  \alpha_+P_X(\Xcal) - \Delta_{\textnormal{conc}}(\Xcal)
\quad\forall\ \Xcal\in\Xset_{x,-}\Bigg\}\geq 1 - \frac{1}{3n_1}.
\end{multline}

Now from this point on, we will assume that the events in~\eqref{eqn:concentration_1},~\eqref{eqn:concentration_2},
 and~\eqref{eqn:concentration_3} all hold, which will occur with probability at least $1-\frac{1}{n_1}$.
We now need to verify that this implies~\eqref{eqn:concentration_Xset} and~\eqref{eqn:concentration_quantile}.

First we verify~\eqref{eqn:concentration_Xset}.
For any $\Xcal\in\widehat{\Xset}_{n_1}$, by definition of $\widehat{\Xset}_{n_1}$ we have
\[\delta \left(1  - \sqrt{\frac{2\log(n_1)}{\delta n_1}}\right) \leq \frac{1}{n_1} \sum_{i=n_0+1}^n \One{X_i\in\Xcal} \\\leq  P_X(\Xcal) +\Delta_{\textnormal{conc}}(\Xcal). \]
Examining the definition of $\delta_-$, we see that this implies $P_X(\Xcal)\geq \delta_-$ when the universal constant $c_\delta$ is chosen appropriately. 
This proves that $\widehat{\Xset}_{n_1}\subseteq \Xset_{x,-}$. Conversely, for any $\Xcal\in\Xset_{x,+}$,
again assuming $c_\delta$ is chosen appropriately, we have
\begin{multline*}
\frac{1}{n_1}  \sum_{i=n_0+1}^n \One{X_i\in\Xcal} \geq  P_X(\Xcal) -\Delta_{\textnormal{conc}}(\Xcal)  \\
 \geq \delta_+ - c\sqrt{\frac{\vc(\Xset)\log^2(n_1)}{n_1}} -\frac{c\log(n_1)}{n_1} \geq \delta\left(1  - \sqrt{\frac{2\log(n_1)}{\delta n_1}}\right) ,
  \end{multline*}and so $\Xcal\in\widehat{\Xset}_{n_1}$. This proves that $\widehat{\Xset}_{n_1}\supseteq \Xset_{x,+}$,
  and therefore,~\eqref{eqn:concentration_Xset} holds whenever the event in~\eqref{eqn:concentration_1} occurs.

 Next we verify~\eqref{eqn:concentration_quantile}. Fix
any $\Xcal\in\Xset_{x,-}$.
By the events in~\eqref{eqn:concentration_1} and~\eqref{eqn:concentration_2}, we have
\begin{multline*}
\sum_{i=n_0+1}^n \One{X_i\in\Xcal, \ |Y_i - \mu(X_i)|>q^*_{P,\mu,\alpha_-}(\Xcal)}\\
\leq n_1\alpha_-P_X(\Xcal) +n_1 \Delta_{\textnormal{conc}}(\Xcal)
\leq n_1\alpha_-\left(\frac{\widehat{N}_{n_1}(\Xcal)}{n_1} + \Delta_{\textnormal{conc}}(\Xcal)\right)+n_1 \Delta_{\textnormal{conc}}(\Xcal)\\
\leq\widehat{N}_{n_1}(\Xcal) \left(\alpha_- +  \frac{2 \Delta_{\textnormal{conc}}(\Xcal) }{\frac{1}{n_1}\widehat{N}_{n_1}(\Xcal)}\right)
\leq\widehat{N}_{n_1}(\Xcal) \left(\alpha_- + \frac{2 \Delta_{\textnormal{conc}}(\Xcal) }{P_X(\Xcal) - \Delta_{\textnormal{conc}}(\Xcal) }\right).
\end{multline*}
Furthermore, by definition of $\alpha_-$, it holds that
\[ \widehat{N}_{n_1}(\Xcal) \left(\alpha_- + \frac{2 \Delta_{\textnormal{conc}}(\Xcal) }{P_X(\Xcal) - \Delta_{\textnormal{conc}}(\Xcal) }\right)
\leq  \widehat{N}_{n_1}(\Xcal)  -\left\lceil\left(1-\alpha+\frac{1}{n_1}\right)\cdot \big(\widehat{N}_{n_1}(\Xcal)+1\big)\right\rceil\]
as long as the constants $c_\alpha,c_\delta$ are chosen appropriately. Combining these calculations, we see that
\[\sum_{i=n_0+1}^n \One{X_i\in\Xcal, \ |Y_i - \mu(X_i)| \leq q^*_{P,\mu,\alpha_-}(\Xcal)}
\geq \left\lceil\left(1-\alpha+\frac{1}{n_1}\right)\cdot \big(\widehat{N}_{n_1}(\Xcal)+1\big)\right\rceil.\]
Since $\qh_{n_1}(\Xcal)$ is defined as
 the $\lceil\left(1-\alpha+\frac{1}{n_1}\right)\cdot (\widehat{N}_{n_1}(\Xcal)+1)\rceil$-th smallest value  in the list $\big\{R_i:n_0+1\leq i\leq n, X_i\in\Xcal\big\}$,
the above bound immediately verifies that
\[\qh_{n_1}(\Xcal) \leq q^*_{P,\mu,\alpha_-}(\Xcal).\]

We can similarly show that, if the events in~\eqref{eqn:concentration_1} and~\eqref{eqn:concentration_3} both hold, then
\begin{multline*}
\sum_{i=n_0+1}^n \One{X_i\in\Xcal, \ |Y_i - \mu(X_i)|\geq q^*_{P,\mu,\alpha_+}(\Xcal)}\\
\geq \widehat{N}_{n_1}(\Xcal) \left(\alpha_+ - \frac{2 \Delta_{\textnormal{conc}}(\Xcal) }{P_X(\Xcal) - \Delta_{\textnormal{conc}}(\Xcal) }\right)> \widehat{N}_{n_1}(\Xcal)\cdot \alpha,\end{multline*}
and by definition of $\qh_{n_1}(\Xcal)$ this is sufficient to establish that
\[\qh_{n_1}(\Xcal) \geq q^*_{P,\mu,\alpha_+}(\Xcal).\]

Therefore, combining everything, we have shown that~\eqref{eqn:concentration_Xset}  and ~\eqref{eqn:concentration_quantile} both hold
whenever the events in~\eqref{eqn:concentration_1},~\eqref{eqn:concentration_2}, and~\eqref{eqn:concentration_3} all hold, 
which occurs 
with probability at least $1-\frac{1}{n_1}$. This completes the proof of the theorem.
\subsubsection{Supporting lemma}
\begin{lemma}\label{lem:vc}
Let $\Xset$ be any collection of measurable subsets of $\R^d$, and let $c:\Xset\rightarrow\R$ be any function.
Fix any function $f:\R^d\times\R\rightarrow\R$, and for each $\Xcal\in\Xset$ define
\[\tilde\Xcal= \left\{(x,y)\in\R^d\times\R : x\in\Xcal\textnormal{ and }f(x,y)>c(\Xcal)\right\}.\]
Then
\[\vc\Big(\{\tilde\Xcal : \Xcal\in\Xset\}\Big)\leq \vc(\Xset)+1.\]
\end{lemma}
\begin{proof}
To see this, suppose $\vc(\{\tilde\Xcal : \Xcal\in\Xset\})=m$. If $m=1$ then the result is trivial, so assume $m\geq 2$.
We can then find $m$ points $(x_i,y_i)\in\R^d\times\R$, for $i=1,\dots,m$, which are shattered by $\{\tilde\Xcal : \Xcal\in\Xset\}$.
Without loss of generality assume that $f(x_m,y_m)=\min_{i=1,\dots,m}f(x_i,y_i)$. We will now show that the set $\{x_1,\dots,x_{m-1}\}$
is shattered by $\Xset$. Fix any subset $I\subseteq\{1,\dots,m-1\}$, and let $\tilde I = I \cup\{m\}$.
Then since $\{\tilde\Xcal : \Xcal\in\Xset\}$ shatters $(x_1,y_1),\dots,(x_m,y_m)$, there must be some $\Xcal\in\Xset$ such
that $(x_i,y_i)\in\tilde\Xcal$ for $i\in\tilde I$ and $(x_i,y_i)\not\in\tilde\Xcal$ for $i\not\in\tilde I$. In particular, taking $i=m\in\tilde I$, we have
\[(x_m,y_m)\in\tilde\Xcal \quad \Rightarrow \quad f(x_m,y_m)>c(\Xcal) \quad\Rightarrow\quad f(x_i,y_i)>c(\Xcal)\textnormal{ for all $i$}.\]
Now, for all $i\in I$, 
\[ i \in\tilde I \quad \Rightarrow \quad (x_i,y_i)\in\tilde\Xcal  \quad \Rightarrow \quad x_i\in\Xcal,\]
and for all $i\in\{1,\dots,m-1\}\backslash I$, we know that $ f(x_i,y_i)>c(\Xcal)$ and therefore
\[i\not\in\tilde I \quad \Rightarrow \quad x_i \not\in\Xcal.\]
Since we can find such a set $\Xcal$ for each subset $I\subseteq\{1,\dots,m-1\}$, this means that $\Xset$ shatters $\{x_1,\dots,x_{m-1}\}$,
and therefore $\vc(\Xset)\geq m-1$, completing the proof.
\end{proof}
 
\subsection{Proof of Corollary~\ref{cor:Xset_VC_upperbd}}
Recall that the oracle interval is given by
\[C^*_P(X_{n+1}) = \mu_P(X_{n+1})\pm  q^*_{\eps,\alpha}\]
where $q^*_{\eps,\alpha}$ is the $(1-\alpha/2)$-quantile of $ f_\eps$.
By Theorem~\ref{thm:Xset_VC_upperbd}, for every $x\in\R^d$ we have
\[\PP{C_{P,\muh_{n_0},\alpha_+,\delta_+}^*(x)\subseteq \Ch_n(x)\subseteq  C_{P,\muh_{n_0},\alpha_-,\delta_-}^*(x)} \geq 1 - \frac{1}{n_1},\]
where $\alpha_+,\alpha_-,\delta_+,\delta_-$ are defined as in the statement of that theorem. Therefore, it must also hold that
\[\PP{C_{P,\muh_{n_0},\alpha_+,\delta_+}^*(X_{n+1})\subseteq \Ch_n(X_{n+1})\subseteq  C_{P,\muh_{n_0},\alpha_-,\delta_-}^*(X_{n+1})} \geq 1 - \frac{1}{n_1},\]
and so with probability at least $1-\frac{1}{n_1}$, we have
\begin{multline*}\leb\big(\Ch_n(X_{n+1})\,\triangle\,C_P^*(X_{n+1})\big) \leq \\ \leb\big(C_{P,\muh_{n_0},\alpha_-,\delta_-}^*(X_{n+1})\backslash C^*_P(X_{n+1})\big)
+\leb\big( C^*_P(X_{n+1})\backslash C_{P,\muh_{n_0},\alpha_+,\delta_+}^*(X_{n+1})\big).
\end{multline*}
Now we bound these two terms.
We can calculate deterministically that
\begin{multline*} \leb\big(C_{P,\muh_{n_0},\alpha_-,\delta_-}^*(X_{n+1})\backslash C^*_P(X_{n+1})\big) \leq\\ |\muh_{n_0}(X_{n+1}) - \mu_P(X_{n+1})| + 2\max\big\{q^*_{P,\muh_{n_0},\alpha_-,\delta_-}(X_{n+1}) - q^*_{\eps,\alpha},0\big\}\end{multline*}
and
\begin{multline*}\leb\big( C^*_P(X_{n+1})\backslash C_{P,\muh_{n_0},\alpha_+,\delta_+}^*(X_{n+1})\big) \leq\\|\muh_{n_0}(X_{n+1}) - \mu_P(X_{n+1})| + 2\max\big\{q^*_{\eps,\alpha} - q^*_{P,\muh_{n_0},\alpha_+,\delta_+}(X_{n+1}),0\big\}.\end{multline*}
Therefore, with probability at least $1-\frac{1}{n_1}$, we have
\begin{multline*}\leb\big(\Ch_n(X_{n+1})\,\triangle\,C_P^*(X_{n+1})\big) \leq 2|\muh_{n_0}(X_{n+1}) - \mu_P(X_{n+1})|\\ 
{}+2\max\big\{q^*_{\eps,\alpha} - q^*_{P,\muh_{n_0},\alpha_+,\delta_+}(X_{n+1}),0\big\}+ 2\max\big\{q^*_{P,\muh_{n_0},\alpha_-,\delta_-}(X_{n+1}) - q^*_{\eps,\alpha},0\big\},
\end{multline*}
so we now need to bound these remaining terms with high probability.

First we bound $|\muh_{n_0}(X_{n+1}) - \mu_P(X_{n+1})|$. Define
\[\widehat\Delta_{n_0} = \EEst{\left(\muh_{n_0}(X) - \mu_P(X)\right)^2}{\muh_{n_0}},\]
which satisfies $\PP{\widehat\Delta_{n_0}\leq \eta_{n_0}}\geq 1-\rho_{n_0}$ by~\eqref{eqn:muh_consistent}.
We have
\begin{multline*}\PP{|\muh_{n_0}(X_{n+1}) - \mu_P(X_{n+1})|> \eta_{n_0}^{1/3}} = \EE{\PPst{|\muh_{n_0}(X_{n+1}) - \mu_P(X_{n+1})|>\eta_{n_0}^{1/3}}{\muh_{n_0}}} \\
\leq \EE{\min\left\{\frac{\EEst{(\muh_{n_0}(X_{n+1}) - \mu_P(X_{n+1}))^2}{\muh_{n_0}}}{\eta_{n_0}^{2/3}},1\right\}}
= \EE{\min\left\{\frac{\widehat\Delta_{n_0}}{\eta_{n_0}^{2/3}},1\right\}} \leq \rho_{n_0} + \frac{\eta_{n_0}}{\eta_{n_0}^{2/3}}.
\end{multline*}
Therefore, with probability at least $1-\frac{1}{n_1} - \rho_{n_0} - \eta_{n_0}^{1/3}$, we have
\begin{multline*}\leb\big(\Ch_n(X_{n+1})\,\triangle\,C_P^*(X_{n+1})\big) \leq 2\eta_{n_0}^{1/3}\\ 
{}+2\max\big\{q^*_{\eps,\alpha} - q^*_{P,\muh_{n_0},\alpha_+,\delta_+}(X_{n+1}),0\big\}+ 2\max\big\{q^*_{P,\muh_{n_0},\alpha_-,\delta_-}(X_{n+1}) - q^*_{\eps,\alpha},0\big\}.
\end{multline*}
Next, by definition we have
\[q^*_{P,\muh_{n_0},\alpha_+,\delta_+}(X_{n+1}) = \sup_{\Xcal\in\Xset : X_{n+1}\in\Xcal, P_X(\Xcal)\geq\delta_+} q^*_{P,\muh_{n_0},\alpha_+}(\Xcal)
\geq q^*_{P,\muh_{n_0},\alpha_+}(\R^d) \geq q^*_{\eps,\alpha_+},\]
where the last step uses  the location
family assumption~\eqref{eqn:loc_family}. Therefore, with probability at least $1-\frac{1}{n_1} - \rho_{n_0} - \eta_{n_0}^{1/3}$, we have
\begin{multline*}\leb\big(\Ch_n(X_{n+1})\,\triangle\,C_P^*(X_{n+1})\big) \leq 2\eta_{n_0}^{1/3} + 2\big(q^*_{\eps,\alpha} - q^*_{\eps,\alpha_+}\big)\\ 
{}+ 2\max\big\{q^*_{P,\muh_{n_0},\alpha_-,\delta_-}(X_{n+1}) - q^*_{\eps,\alpha},0\big\}.
\end{multline*}
We now address the last term.
By definition we have
\[q^*_{P,\muh_{n_0},\alpha_-,\delta_-}(X_{n+1}) = \sup_{\Xcal\in\Xset : X_{n+1}\in\Xcal, P_X(\Xcal)\geq\delta_-} q^*_{P,\muh_{n_0},\alpha_-}(\Xcal)
\leq \sup_{\Xcal\in\Xset : P_X(\Xcal)\geq\delta_-} q^*_{P,\muh_{n_0},\alpha_-}(\Xcal).\]
By the location
family assumption~\eqref{eqn:loc_family} we can see that, for any $\Xcal$,
\[q^*_{P,\muh_{n_0},\alpha_-}(\Xcal)\leq \min_{0< \alpha' < \alpha_-}\left\{q^*_{\eps,\alpha_--\alpha'} + \textnormal{\begin{tabular}{c}the $(1-\alpha')$-quantile
of $|\muh_{n_0}(X) - \mu_P(X)|$ \\conditional on $\muh_{n_0}$ and on $X\in\Xcal$\end{tabular}}\right\}.\]
And, for any $\Xcal$ with $P_X(\Xcal)\geq\delta_-$, this last quantile is bounded by
\[\sqrt{\frac{\EEst{(\muh_{n_0}(X) - \mu_P(X))^2}{\muh_{n_0},X\in\Xcal}}{\alpha'}}\leq \sqrt{\frac{\widehat\Delta_{n_0}}{\alpha'\delta_-}}.\]
Therefore, choosing $\alpha' = \eta_{n_0}^{1/3}$,
\[q^*_{P,\muh_{n_0},\alpha_-,\delta_-}(X_{n+1})  \leq q^*_{\eps,\alpha_--\eta_{n_0}^{1/3}} +  \sqrt{\frac{\widehat\Delta_{n_0}}{\eta_{n_0}^{1/3}\delta_-}}\leq q^*_{\eps,\alpha_--\eta_{n_0}^{1/3}} +  \eta_{n_0}^{1/3}\delta_-^{-1/2}\,\]
where the last bound holds with probability at least $1-\rho_{n_0}$ by~\eqref{eqn:muh_consistent}.
Combining everything, 
 with probability at least $1-\frac{1}{n_1} - 2\rho_{n_0} - \eta_{n_0}^{1/3}$, we have
\[\leb\big(\Ch_n(X_{n+1})\,\triangle\,C_P^*(X_{n+1})\big) \leq 2\eta_{n_0}^{1/3} + 2\big(q^*_{\eps,\alpha_--\eta_{n_0}^{1/3}} - q^*_{\eps,\alpha_+}\big)
+ 2 \eta_{n_0}^{1/3}\delta_-^{-1/2}.\]
Finally, by our assumptions~\eqref{eqn:loc_family} on the density $f_\eps$ and the definition of $q^*_{\eps,\cdot}$, for any $\alpha'< \alpha''\in[0,1]$ we have
\[\frac{1}{2}(\alpha''-\alpha') = \int_{t=q^*_{\eps,\alpha''}}^{q^*_{\eps,\alpha'}}f_\eps(t)\;\mathsf{d}t \geq f_\eps(q^*_{\eps,\alpha'})\cdot \big(q^*_{\eps,\alpha'} - q^*_{\eps,\alpha''}\big).\]
Therefore,
\[q^*_{\eps,\alpha_--\eta_{n_0}^{1/3}} - q^*_{\eps,\alpha_+} \leq \frac{\alpha_+ - \big(\alpha_- - \eta_{n_0}^{1/3} \big)}{2f_\eps(q^*_{\eps,\alpha_--\eta_{n_0}^{1/3}} )},\]
which completes the proof for constants $c,c'$ chosen appropriately.

\bibliographystyle{plainnat}
\bibliography{bib}

\begin{thebibliography}{13}
\providecommand{\natexlab}[1]{#1}
\providecommand{\url}[1]{\texttt{#1}}
\expandafter\ifx\csname urlstyle\endcsname\relax
  \providecommand{\doi}[1]{doi: #1}\else
  \providecommand{\doi}{doi: \begingroup \urlstyle{rm}\Url}\fi

\bibitem[Bahadur and Savage(1956)]{bahadur1956nonexistence}
Raghu~R Bahadur and Leonard~J Savage.
\newblock The nonexistence of certain statistical procedures in nonparametric
  problems.
\newblock \emph{The Annals of Mathematical Statistics}, 27\penalty0
  (4):\penalty0 1115--1122, 1956.

\bibitem[Blumer et~al.(1989)Blumer, Ehrenfeucht, Haussler, and
  Warmuth]{blumer1989learnability}
Anselm Blumer, Andrzej Ehrenfeucht, David Haussler, and Manfred~K Warmuth.
\newblock Learnability and the {V}apnik--{C}hervonenkis dimension.
\newblock \emph{Journal of the ACM}, 36\penalty0 (4):\penalty0 929--965, 1989.

\bibitem[Cai et~al.(2014)Cai, Low, and Ma]{cai2014adaptive}
T~Tony Cai, Mark Low, and Zongming Ma.
\newblock Adaptive confidence bands for nonparametric regression functions.
\newblock \emph{Journal of the American Statistical Association}, 109\penalty0
  (507):\penalty0 1054--1070, 2014.

\bibitem[Donoho(1988)]{donoho1988one}
David~L Donoho.
\newblock One-sided inference about functionals of a density.
\newblock \emph{The Annals of Statistics}, 16\penalty0 (4):\penalty0
  1390--1420, 1988.

\bibitem[Dudley and Norvai{\v{s}}a(2011)]{dudley2011concrete}
Richard~M Dudley and Rimas Norvai{\v{s}}a.
\newblock \emph{Concrete functional calculus}.
\newblock Springer, 2011.

\bibitem[Koltchinskii(2011)]{koltchinskii2011oracle}
Vladimir Koltchinskii.
\newblock \emph{Oracle Inequalities in Empirical Risk Minimization and Sparse
  Recovery Problems: Ecole d'Et{\'e} de Probabilit{\'e}s de Saint-Flour
  XXXVIII-2008}, volume 2033.
\newblock Springer Science \& Business Media, 2011.

\bibitem[Lei and Wasserman(2014)]{lei2014distribution}
Jing Lei and Larry Wasserman.
\newblock Distribution-free prediction bands for non-parametric regression.
\newblock \emph{Journal of the Royal Statistical Society: Series B (Statistical
  Methodology)}, 76\penalty0 (1):\penalty0 71--96, 2014.

\bibitem[Lei et~al.(2018)Lei, G'Sell, Rinaldo, Tibshirani, and
  Wasserman]{lei2018distribution}
Jing Lei, Max G'Sell, Alessandro Rinaldo, Ryan~J Tibshirani, and Larry
  Wasserman.
\newblock Distribution-free predictive inference for regression.
\newblock \emph{Journal of the American Statistical Association}, 113\penalty0
  (523):\penalty0 1094--1111, 2018.

\bibitem[Papadopoulos(2008)]{papadopoulos2008inductive}
Harris Papadopoulos.
\newblock Inductive conformal prediction: Theory and application to neural
  networks.
\newblock In \emph{Tools in artificial intelligence}. InTech, 2008.

\bibitem[Papadopoulos et~al.(2002)Papadopoulos, Proedrou, Vovk, and
  Gammerman]{papadopoulos2002inductive}
Harris Papadopoulos, Kostas Proedrou, Volodya Vovk, and Alex Gammerman.
\newblock Inductive confidence machines for regression.
\newblock In \emph{European Conference on Machine Learning}, pages 345--356.
  Springer, 2002.

\bibitem[Shafer and Vovk(2008)]{shafer2008tutorial}
Glenn Shafer and Vladimir Vovk.
\newblock A tutorial on conformal prediction.
\newblock \emph{Journal of Machine Learning Research}, 9\penalty0
  (Mar):\penalty0 371--421, 2008.

\bibitem[Vovk(2012)]{vovk2012conditional}
Vladimir Vovk.
\newblock Conditional validity of inductive conformal predictors.
\newblock In \emph{Asian conference on machine learning}, pages 475--490, 2012.

\bibitem[Vovk et~al.(2005)Vovk, Gammerman, and Shafer]{vovk2005algorithmic}
Vladimir Vovk, Alex Gammerman, and Glenn Shafer.
\newblock \emph{Algorithmic learning in a random world}.
\newblock Springer Science \& Business Media, 2005.

\end{thebibliography}
\end{document}